\documentclass[reqno,12pt]{amsart}
\usepackage{graphicx}
\usepackage{amssymb}
\usepackage{amsmath}
\usepackage{amsthm}
\usepackage{paralist}
\usepackage{mathtools}
\usepackage{faktor} 
\usepackage{tikz}
\usepackage{csquotes}
\usepackage{caption}
\usetikzlibrary{babel,matrix,arrows,decorations.pathmorphing}
\usepackage{tikz-cd}
\usepackage[table,dvipsnames]{xcolor}
\usepackage{algorithm}
\usepackage[noend]{algpseudocode}

\usepackage{rotating}

\tikzset{commutative diagrams/.cd}
\usepackage[margin = 3cm]{geometry}
\newtheorem{Lemma}{Lemma}[section]
\newtheorem{Corollary}[Lemma]{Corollary}
\newtheorem{Proposition}[Lemma]{Proposition}
\newtheorem{Theorem}[Lemma]{Theorem}

\theoremstyle{definition}
\newtheorem{Example}[Lemma]{Example}
\newtheorem{Definition}[Lemma]{Definition}
\newtheorem{Claimm}[Lemma]{Claim}
\newtheorem{Possibility}[Lemma]{Realization}
\newtheorem{Remark}[Lemma]{Remark}
\newtheorem*{convention}{Convention}
\newtheorem*{question}{Question}

\newtheorem*{Claim}{Claim}

\newcommand{\discgr}[1]{\text{$A_{#1}$}}
\newcommand{\discg}{\text{$A_{\textsc{NS}(X)}$}}
\newcommand{\NS}{\textsc{NS$(X)$}}
\newcommand{\NSY}{\textsc{NS$(Y)$}}

\newcommand{\dul}[1]{#1^\vee}

\newcommand{\bil}[2]{\left(#1.#2  \right)}
\newcommand{\bilemp}{\bil{\text{-}}{\text{-}}}

\newcommand{\mC}{\mathbb C}
\newcommand{\mZ}{\mathbb Z}

\newcommand{\hyz}{H^2(Y,\mZ)}

\DeclareMathOperator{\romantwo}{\MakeUppercase{\romannumeral 2}}

\usepackage[style=alphabetic]{biblatex}
\DeclareMathOperator{\sign}{sign}
\DeclareMathOperator{\Aut}{Aut}
\DeclareMathOperator{\rk}{rk}
\DeclareMathOperator{\id}{id}
\DeclareMathOperator{\GL}{GL}

\usepackage{hyperref}
\hypersetup{
    colorlinks=true,
    linkcolor=Blue,
    filecolor=magenta,      
    urlcolor=cyan,
    citecolor=RedViolet,
    pdftitle={Extensions of maximal symplectic groups acting on
the superspecial K3 surface in characteristic p}
    }

\date{} 
\author{Sophie E. Friesen}
\address{Institut f\"{u}r Algebraische Geometrie, Leibniz Universit\"{a}t Hannover, Welfengarten 1, 30167 Hannover, Germany.}
\email{friesen@math.uni-hannover.de}
\begin{document}
\title[Extensions of maximal symplectic groups on superspecial K3s]{Extensions of maximal symplectic groups acting on the superspecial $K3$ surface in characteristic $p$
}
\begin{abstract}
   \noindent We classify all possible finite groups acting faithfully on the superspecial $K3$ surface which contain the maximal symplectic groups, as classified by Ohashi and Schütt.
\end{abstract}
\maketitle

\section{Introduction}
Automorphism groups of $K3$ surfaces have been an active area of research for the last decades. Starting with foundational work by Nikulin \parencite{NikulinK3}, a central question was the following:
\begin{question}
    Let $G$ be a finite group. Is there a $K3$ surface $X$ over an algebraically closed field $k$ such that $G$ embeds into $\text{Aut}(X)$?
\end{question}
\noindent In order to study finite groups of automorphisms of $K3$ surfaces, it is helpful to first focus on \textit{symplectic} automorphisms, i.e.\ automorphisms that leave a nowhere vanishing two-form invariant. Let $X$ be a $K3$ surface. We denote the group of symplectic automorphisms as $\Aut_s(X)$. The symplectic automorphisms contained in $G\subset \text{Aut}(X)$ form a normal subgroup, denoted by $G_s$. A finite symplectic group $G_s\subset \Aut_s(X)$ is called \textit{maximal} if there is no finite subgroup $G_s'$ such that $G_s\subsetneq G_s'\subsetneq\Aut_s(X)$. A finite group $G$ whose symplectic subgroup is $G_s$ is called an \textit{extension} of $G_s$ and $ [G:G_s]$ is called the \textit{non-symplectic index} of $G$ (cf. \parencite{JangIndex}). An extension is \textit{non-trivial} if $[G:G_s]\geq 2$. \\\
Over the complex numbers, Mukai \parencite{Mukai1988} classifies all the possible finite groups which can occur as $G_s$.
This has been extended to a classification of the finite groups with saturated symplectic subgroups by Brandhorst, Hashimoto and Hofmann \parencite{brandhorst-hashimoto}\parencite{Brandhorst_Hofmann_2023}. 
These classifications need not hold over an algebraically closed field $k$ of characteristic $p$ due to the existence of \textit{supersingular} $K3$ surfaces, as has been noted by Dolgachev and Keum \parencite{DolgachevKeum2009}. Indeed, the symplectic automorphism group of these $K3$ surfaces can contain finite subgroups that are not realized over the complex numbers, see for instance \parencite[Theorem 1.1]{ohashi2024finite}. The following is the main result of this paper.
\begin{Theorem}\label{mainmainresult}
    Let $X$ be a $K3$ surface over an algebraically closed field of characteristic $p >11$. Assume that $G\subset \Aut(X)$ is a finite group such that $G_s$ does not appear over the complex numbers (i.e.\ cannot be found in the classification of Mukai \parencite{Mukai1988}). Then $G = G_s$.

\end{Theorem}
\noindent This result is in contrast to the complex case, where each maximal group has an action on some algebraic complex singular $K3$ surface that possesses a non-trivial extension \parencite{brandhorst-hashimoto}. The proof of Theorem \ref{mainmainresult} will be discussed in Section \ref{mainmainresultproof}. An important input for this proof will be Theorem \ref{TameCase}(i) as well as the results of \parencite{ohashi2024finite}. \\
For every prime number $p$, among the supersingular $K3$ surfaces there is a unique one with Artin invariant 1 called the \textit{superspecial} $K3$ surface. More details on this are discussed in Section \ref{K3s}. 
The classification of finite symplectic actions on the superspecial $K3$ surface was recently completed by Sch\"utt and Ohashi in \parencite{ohashi2024finite}, with further insights on the case of higher Artin invariants given by Wang and Zheng in \parencite{wang2024finitegroupssymplecticautomorphisms}. The superspecial $K3$ surface has very exceptional properties when considering symplectic automorphisms. Namely, in a given characteristic $p$, all finite symplectic groups $G_s$ with $p \nmid \vert G_s\vert$ (i.e.\ \textit{tame groups}) which do not appear in the classification of Mukai act exclusively on the superspecial $K3$ surface \parencite[Theorem 1.1]{ohashi2024finite}. It is therefore a natural question to study the possible extensions of these actions on superspecial $K3$ surfaces.\\
 In this work we aim to extend the results of Ohashi and Sch\"utt to finite groups which contain the maximal symplectic subgroups using methods derived from lattice theory. This mirrors the works of Brandhorst and Hashimoto over $\mathbb C$ \parencite{brandhorst-hashimoto}.                                                                  
\begin{Theorem}\label{TameCase}
    Let $k$ be an algebraically closed field of characteristic $p> 11$, $X$ the superspecial $K3$ surface over $k$ and $G_s\subset\Aut_s(X)$ be a finite maximal group. Then one of the following holds 
    \begin{itemize}
        \item[(i)] $G_s$ does not appear over the complex numbers (i.e.\ is not in the classification of Mukai \parencite{Mukai1988}) and has no non-trivial extension.
        \item[(ii)] $G_s$ is one of the maximal groups over $\mathbb C$ classified in \parencite{Mukai1988} and, depending on $p$, there is an extension $G\subset \Aut(X)$ of $G_s$ with non-symplectic index:
        \begin{center}{\rowcolors{2}{white}{lightgray}
\begin{tabular}{ |c|c|c|c|c| } 
\hline
$G_s$& $\vert G :G_s\vert= 1$&$\vert G :G_s\vert = 4$ & $\vert G :G_s\vert =6$ & $\vert G : G_s \vert = 2$\\
\hline
$T_{48}$ &  $p  \equiv 1 \mod 24$ & - & $p\equiv 5 \mod 6$ & otherwise\\
$N_{72}$ & $p  \equiv 1 \mod 24$ & $p\equiv 3 \mod 4$ & - & otherwise\\
$M_9$ & $p  \equiv 1 \mod 24$ & - & $p\equiv 5 \mod 6$& otherwise\\
$\mathfrak S_5$ & $p \equiv 1,49 \mod 120$ &- &-& otherwise\\
$L_2(7)$ &  $p  \equiv 1,9,25 \mod 56$& $p\equiv 3 \mod 4$ & -& otherwise\\
$H_{192}$ & $p  \equiv 1 \mod 24$ & -&-& otherwise\\
$T_{192}$ & $p  \equiv 1 \mod 24$ & - & $p\equiv 5 \mod 6$& otherwise\\
$\mathfrak A_{4,4}$ & $p  \equiv 1 \mod 24$& $p\equiv 3 \mod 4$ & -& otherwise\\
$\mathfrak A_{6}$ &  $p
\equiv 1,49 \mod 120$ & $p\equiv 3 \mod 4$ & -& otherwise\\
$F_{384}$ & $p\equiv 1 \mod 8$& $p\equiv 3 \mod 4$ & -& otherwise\\
$M_{20} $& $p  \equiv1,9 \mod 40$&$p\equiv 3 \mod 4$ & -& otherwise\\
\hline
\end{tabular}}
\captionof{table}{Non-trivial extensions of the Mukai groups}
\end{center}
Furthermore, every extension $G'\subset \Aut(X)$ of $G_s$ in characteristic $p$ the index $[G’:G_s]$ divides $[G:G_s]$
for the respective congruence class of $p$ above.

    \end{itemize}
\end{Theorem}
\noindent 
We note that Theorem \ref{TameCase} holds even more generally. Namely, the groups that are maximal over the complex numbers act on the superspecial $K3$ surface in any characteristic $p\geq 3$, but this symplectic action is not always maximal. The table in Theorem \ref{TameCase}(ii) also holds for primes where the group action of $G_s$ is not maximal on the superspecial $K3$ surface, as long as $p \nmid \vert G_s\vert$, i.e.\ the action of $G_s$ is \textit{tame}.
More details on this will be discussed in Section \ref{tame}, where the proof of Theorem \ref{TameCase} can be found in Section \ref{mainmainresultproof}. The proof of part (i) will rely on a lattice theoretic criterion which obstructs the existence of a non-trivial extension (see Lemma  \ref{AusschlussDuallyPrimitive} as well as Remark \ref{explanationduallyprim}). For the proof of part (ii), we note that group actions from complex $K3$ surfaces can descend to group actions on the superspecial $K3$ surface. 
Namely, assume that  $G\subset \Aut(Y)$ is a finite group and $Y$ is a  complex $K3$ surface which has a good reduction modulo $p$ and whose reduction is isomorphic to the superspecial $K3$ surface in characteristic $p$. Then if the action of $G$ also has a good reduction modulo $p$, clearly $G$ must be contained in the automorphism group of the superspecial $K3$ surface $X$ in characteristic $p$. The goal of Section \ref{MukaiGroups} is 
to show that a weaker converse statement holds. Namely, 
a finite group $G$ can only embed into $\Aut(X)$ such that $G_s$ is a Mukai group if there is a complex $K3$ surface $Y$ such that $Y$ 'almost' has superspecial reduction modulo $p$, $G_s\subset \Aut_s(Y)$ and there is an extension of $G_s$ on $Y$ with the same non-symplectic index as $G$. The rigorous formulation of this goal can be found in Theorem \ref{CorrespondenceComplexSupSpe}. This equivalence also explains, that in certain characteristics no non-trivial extensions can occur; in essence, we may see from the results of \parencite{brandhorst-hashimoto} that no suitable complex $K3$ surface $Y$ exists.
\\
Theorem \ref{TameCase} showcases that there are multiple instances of maximal symplectic groups that do not have a non-trivial extension. We exhibit a condition on $p$ such that every maximal symplectic group has a non-trivial extension.
\begin{Theorem}\label{alotofcongruences}
    Let $X$ be the superspecial $K3$ surface in characteristic $p > 11$. Then every finite maximal symplectic group $G_s\subset \Aut_s(X)$ has a non-trivial extension if and only if 
    $p$ satisfies \[p\equiv -1,-(11^2),-(13^2),-(17^2),-(19^2),-(23^2)\mod 840.\]
\end{Theorem}
\noindent This will be proven in Section \ref{mainmainresultproof}.\\
Furthermore, we will classify extensions of maximal symplectic actions in characteristic $p\leq 11$. In this setting \textit{wild} automorphism groups (i.e.\ $p\mid \vert G_s\vert$) appear. For every non-trivial extension, we will exhibit a model of an extension with the highest possible non-symplectic index in Section \ref{Char23}. We obtain the following result.
\begin{Theorem}\label{wildcase}
     Let $k$ be an algebraically closed field of characteristic $p\leq 11$ and $X$ the superspecial $K3$ surface over $k$. Let $G\subset \Aut(X)$. Assume that the symplectic subgroup $G_s$ of $G$ is a maximal group in the given characteristic $p$. Then $G$ is a subgroup of one the following groups, denoted by $G_s.n$ (with another isomorphic description if one is available), where $n$ is the non-symplectic index:
\end{Theorem}  
    \begin{itemize}
        \item [$(p=2)$]$(M_{21}.2_1).3 (\cong \text{PGU}(3,\mathbb F_{2^2}).2)$ of order $ 120,960= 2^7.3^3.5.7$  \\
        $\text{Aut}(\mathfrak S_6).1$ of order $1,440= 2^5.3^2.5$
        \item[$(p=3)$] $ \text{PSU}(4,\mathbb F_{3^2}).4 (\cong \textrm{PGU}(4,\mathbb F_{3^2}))$ of order $13,063,680 = 2^9.3^6.5.7$\\
        $(2^2.\mathfrak A_{4,4}).2$ of order $2,304= 2^8.3^2$ \\
        $(2^4:\mathfrak S_{3,3}).4$ of order $2,304 =2^8.3^2 $ 
        \item[$(p=5)$] $ \text{PSU}(3,\mathbb F_{5^2}) . 6 (\cong \text{PGU}(3,\mathbb F_{5^2}): 2)$ of order $75,600= 2^5.3^3.5^3.7$    \\
        $\mathfrak A_8.1$ of order $20,160 = 2^6.3^2.5.7$\\
        $(2^4:(3\times \mathfrak A_5):2).1$ of order $5,760= 2^7.3^2.5$ 
        \item[$(p=7)$]  $(L_3(4).2).1$ of order $40,320 = 2^7.3^2.5.7$ \\
        $(2^4:\mathfrak A_7).1$ of order $40,320 = 2^7.3^2.5.7$
        \item[$(p=11)$] $M_{22}.1$ of order $443,520 =2^7.3^2.5.7.11 $ \\
        $M_{11}.1$ of order $7,920 = 2^3.3^3.5.11$ \\
        $(2^4:\mathfrak S_{3,3}).1$ of order $576 = 2^6.3^2$
        \end{itemize}
\vspace{8 pt}
\noindent For further reference, we fix the notation of group and lattice theoretic expressions, following the conventions in \parencite{ohashi2024finite}. 
\begin{convention}
    We use the genus symbol from Conway and Sloane \parencite[Chapter 15]{Conway_Sloane} including the refinement from \parencite{Allcock}. Furthermore for finite groups $G,H$ we write:
    \indent $G:H$ - the semi direct product of subgroups $G$, $H$, where $G$ is normal \newline
    \indent $G.H$ - a group $A$ with a normal subgroup $G$ such that $A/G\cong H$ \newline
    \indent $\mathfrak S_{m,n} \coloneqq \mathfrak S_m \times \mathfrak S_n$ \newline 
    \indent $\mathfrak A_{m,n} \coloneqq \mathfrak S_{m,n} \cap \mathfrak A_{m+n}$ \newline 
    \indent $n^e \coloneqq (\mathbb Z/n\mathbb Z)^e$
    \newline 
    \indent $D_n$ - the dihedral group of order $2n$\\
All $K3$ surfaces we will consider will be over algebraically closed fields.
\end{convention}

\noindent\textbf{Acknowledgments.} The author would like to thank their advisor Matthias Schütt for suggesting this project, many important discussions and feedback, Marie Roth for a very helpful discussion on Clifford's theorem, Stevell Muller for important input regarding lattice theory and feedback as well as Dominique Mattei and Simon Brandhorst for detailed comments. This work was funded by the Deutsche Forschungsgemeinschaft (DFG, German Research Foundation) — RTG2965 — Project number 512730679.\\
  \begin{center}      \includegraphics[width=0.75\textwidth]{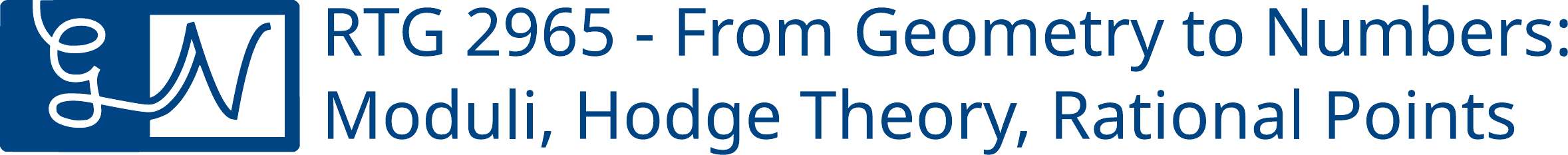}
  \end{center}

\tableofcontents

\section{\texorpdfstring{Preliminaries}{Preliminaries}}
\subsection{Lattice theory}
\subsubsection{Definitions}
Let $L$ be an integral lattice, i.e.\ a finitely generated free $\mathbb Z$-module with a nondegenerate, symmetric bilinear form $\bilemp\ L \times L \rightarrow \mathbb Z$. Unless specified, all lattices $L$ are assumed to be integral and\textit{ even}, i.e.\ $\bil{x}{x}\in 2 \mathbb Z$ for every $x \in L$. The \textit{signature} $(n_+,n_-)$ of $L$ is given by the number of positive and negative eigenvalues of a Gram matrix of $L$. When rescaling the bilinear form $\bil{x}{x}$ of $L$ with a non-zero integer $a$, we obtain a new lattice that we denote as $L[a]$. Note that $L[a]$ may have a different signature if $a$ is negative.\\
A $\mathbb Z$-module automorphism of $L$ that respects the bilinear form is called an \textit{isometry}. The group of isometries of $L$ is denoted by $O(L)$.\\
The \textit{dual lattice} $\dul{L}$ is defined as 
\[
\dul{L} \coloneqq \{x \in L\otimes \mathbb Q \ \vert \ \bil{x}{L}\subseteq \mathbb Z   \}.
\]
The dual lattice is not necessarily integral as the bilinear form does not need to map to the integers. A lattice is called \textit{unimodular} if $L^\vee \cong L$. Since $L^\vee$ contains $L$, we may consider the quotient group
\[
 \discgr{L} \coloneqq \faktor{\dul{L}}{L}.
\]
We call the finite abelian group $A_L$ the \textit{discriminant group} of $L$. The order of $A_L$ is equal to the absolute value of the determinant of a Gram matrix of $L$. We denote the determinant of a Gram matrix of $L$ as $\det(L)$. We define the \textit{length} $l(A_L)$ of $L$ as the smallest number of generators of $\discgr{L}$. The length $l(A_L)$ is always bounded by the rank of $L$. Furthermore, for any prime $p$ we can define $l_p(A_L)$ to be the smallest number of generators of the Sylow $p$-subgroup $(\discgr{L})_p$ of the discriminant group.
The natural map
\[
 \phi\colon  \dul{L} \rightarrow \discgr{L}
\]
sends every element $x \in \dul L$ to its associated equivalence class $\Bar{x}$ in $\discgr L$. 
The bilinear form $\bilemp$ extends to $L^\vee$ and induces a quadratic form $q_L\colon\discgr{L}  \rightarrow \mathbb Q /2\mathbb Z$, $q_L(z)\coloneqq (z.z)\mod 2\mathbb Z$. 
We say that $f\colon \discgr{L} \rightarrow \discgr{L}    $ is an \textit{automorphism} of the discriminant group if $f$ is an automorphism of groups and preserves the quadratic form $q_L$. Every isometry $g$ of $L$ induces an automorphism of $A_L$, that we will denote as $g_A$.\\
Let $M$ be a sublattice of $L$. We say that $M$ is \textit{primitive} in $L$ if $L/M$ is torsion free. For $x_1,\dots,x_l\in L$ we denote the sublattice spanned by the $x_i$ as $\langle x_1\dots x_l\rangle$. This sublattice is not necessarily primitive.\\
If $M = \langle x \rangle$ for $x \in L$ is primitive in $L$, we say that $x$ is primitive in $L$. 
\begin{Definition}
    We say that an element $x$ of the lattice $L$ is \textit{dually primitive}, if $x$ is primitive in $\dul{L}$ or, equivalently, $\bil{x}{L}= \mathbb Z$. This property is also sometimes called being \textit{1-divisible}.
\end{Definition}
\begin{Example}
A simple example is given by considering the lattice defined by the Gram matrix 
\[
\begin{pmatrix}
    2 & 0 \\ 0 & -2
\end{pmatrix}.
\]
Then the two generators $x,y$ are primitive in the lattice, but not dually primitive since $\frac{1}{2} x$ and $\frac{1}{2}y$ are the generators of the dual lattice.\\
On the other hand, if we consider the lattice defined by 
\[\begin{pmatrix}
    2 & 0 \\ 0 & -1
\end{pmatrix}
,\] the second generator is dually primitive.
\end{Example}
 \noindent For two lattices $L_1$ and $L_2$, the \textit{orthogonal direct sum} is written as $L_1\oplus L_2$.
For a sublattice $L_1 \subset L$, the \textit{orthogonal complement} of $L_1$ in $L$ is written as
\[
L_1^\perp \coloneqq \{x \in L \mid (x.y)=0 \ \text{for every } y\in L_1   \}.
\]
Since we assume that $L_1$ is a lattice, and thus nondegenerate, the orthogonal direct sum of $L_1$ and $L_1^\perp$ has the same rank as $L$ and $L/(L_1\oplus L_1^\perp)$ is finite.
\subsubsection{Primitive extensions of lattices}
Given two lattices $L_1$ and $L_2$, we are interested in lattices $L\supset L_1\oplus L_2$ such that both $L_i$ are primitive in $L$ and $\rk(L) =\rk(L_1)+\rk(L_2)$. Such a lattice is called a \textit{primitive extension} of $L_1$ and $L_2$ (see \parencite{Nikulin_1980} or \parencite[section 2.2]{Brandhorst_Hofmann_2023}). 
Primitive extensions can be identified in the discriminant group in the following way. There is a natural sequence of embeddings
\[
L_1 \oplus L_2 \hookrightarrow L \hookrightarrow L^\vee \hookrightarrow (L_1\oplus L_2)^\vee \stackrel{(1)}{\cong} L_1^\vee \oplus L_2^\vee.
\]
This implies that we may identify $L$ with the subgroup $L/(L_1\oplus L_2)\subset A_{L_1\oplus L_2}$. Since we take $L$ to be an even lattice, the quadratic form $q_{L_1\oplus L_2}$ is trivial along the subgroup $L/(L_1\oplus L_2)$.\\
Due to the isomorphism (1), we can project $L/(L_1\oplus L_2)$ into $A_{L_1}$ and $A_{L_2}$ respectively. These maps are injective since the $L_i$ are assumed to be primitive in $L$. We denote the images of the projections by $H_{L_i}\subset A_{L_i}$.
There are group isomorphisms
\[
\begin{tikzcd}
H_{L_1} \arrow[rr, "g", bend right, shift right=2] \arrow[r, "\cong"] & L/(L_1\oplus L_2) \arrow[r, "\cong"] & H_{L_2}
\end{tikzcd}
\]
The map $g $ is also called the \textit{gluing map}. Then, since we know that the quadratic form on $L/(L_1 \oplus L_2)$ is isotropic, we have that $q_{L_1}(x)= -q_{L_2}(g(x))$ for every $x\in A_{L_1}$. Up to sign, the determinant of $L$ can be calculated in the following way
\[
\vert\det(L)\vert = \vert A_{L_1}/H_{L_1}\vert \cdot \vert A_{L_2}/H_{L_2}\vert .
\]
Consider the lattices $L_1$ and $L_2$ such that there exists a primitive extension $L$. We are interested in \textit{extending} an isometry $f_1\in O(L_1)$ to an isometry of $L$, i.e.\ an isometry $f \in O(L)$ such that $f_{\vert L_1} = f_1$.\\
An isometry $f_1\in O(L_1)$ can be extended to a (unique) isometry $f \in O(L)$ if and only if there exists an isometry $f_2\in O(L_2)$ such that $g((f_1)_A\vert_{H_1}) = (f_2)_A\vert_{H_2} $. \\

\subsubsection{The genus symbol}
Let $L_1$ and $L_2$ be two lattices. We say that $L_1$ and $L_2$ lie in the same \textit{genus} when $L_1\otimes \mathbb R \cong L_2\otimes \mathbb R$ and $L_1\otimes \mathbb Z_q \cong L_2\otimes \mathbb Z_q$ for every prime $q$. Here, $\mathbb Z_q$ is the ring of $q$-adic integers. The condition $L_1\otimes \mathbb R \cong L_2\otimes \mathbb R$ is equivalent to $L_1$ and $L_2$ having the same signature.\\
It is clear that two isomorphic lattices $L_1$ and $L_2$ lie in the same genus. But in general, the converse does not hold. \\
Let $L$ be a lattice of rank $n$ with signature $(n_+,n_-)$ and determinant $d$. The \textit{genus symbol} is a shorthand notation that encodes the behaviour of $L\otimes\mathbb Z_q$ for every prime $q$ and the signature of $L$. Very importantly, this information is equivalent to describing the quadratic form $q_L$ of the discriminant group of $L$ and the signature of $L$. The genus symbol was first introduced by Conway and Sloane \parencite{Conway_Sloane} and we briefly recall its construction.\\
Let $q\geq 3$ be a prime. Then every Gram matrix of $L\otimes \mathbb Z_q$ is diagonalizable. We may write this diagonal matrix as
\[
M = \begin{pmatrix}
    a_1 q^{e_1} & & &\\
    & a_2 q^{e_2} & & \\
    & & \ddots &\\
    & & & a_{n} q^{e_n}
\end{pmatrix}
\]
where $e_i\leq e_{i+1}$ and $a_i\in \mathbb Z_q^*$.  
We may decompose $M$ into the block matrix
\[
M = \bigoplus\limits_{i} q^iM_i
\]
where each $M_i$ is a diagonal matrix with entries $a_{i_1},\dots, a_{i_{n_i}}$ for which $i=e_{i_1}=\dots= e_{i_{n_i}}$. The determinant of $M_i$ depends on the explicit choice of basis. A different choice of basis varies the determinant of $M_i$ by a square in $\mathbb Z_q^*$. The \textit{$q$-adic genus symbol} is written as
\[
    1^{\epsilon_0n_0}q^{\epsilon_1n_1}\cdots \left(q^e \right)^{\epsilon_en_e},
    \]
where $\epsilon_i = +1$ when the determinant of $M_i$ is in $ (\mathbb Z_q^*)^2$ and $\epsilon_i = -1$ otherwise. 
The $q$-adic genus symbol uniquely determines the isometry class of the $q$-adic lattice $L\otimes \mathbb Z_q$.\\
The 2\textit{-adic genus symbol} requires more care since $L\otimes \mathbb Z_2$ is generally not diagonizable. In other words, we may still decompose a Gram matrix $M$ into the block matrix
\[
M = \bigoplus\limits_{i} p^iM_i,
\]
but the blocks $M_i$ will generally not be diagonal.
The definition requires the notion of \textit{oddity}. The oddity $t_i$ of the 2-adic genus at $2^i$ is defined as $\romantwo$ if all diagonal entries are even and $t_i\equiv \text{tr}(M_i)\mod 8$ otherwise. We write the genus symbol at 2 as
\[
    1_{(t_0)}^{\epsilon_0n_0}2_{(t_1)}^{\epsilon_1n_1}\cdots \left(2^e \right)_{(t_e)}^{\epsilon_en_e},
    \]
where $\epsilon =+1$ if $\det(M_i)\equiv 1,7\mod 8$ and $\epsilon = -1$ if $\det(M_i)\equiv 3,5\mod 8$. The genus symbol at 2 is not unique due to the occurrence of sign walking and oddity fusion. More work on this is found in \parencite{Conway_Sloane} as well as \parencite{Allcock}.\\
The \textit{genus symbol} of an even lattice $L$ with determinant $d$ and signature $(n_+,n_-)$ is written as 
\[
\romantwo_{n_+,n_-}(G_q)_{q \text{ prime }},
\]
where $G_q$ is the $q$-adic genus symbol of $L$. For readability, we may omit the parts $1^{\epsilon_0}_{n_0}$ for each $q$, which implies that for every prime $q\nmid d$ the genus symbol at $q$ can be omitted from our notation. \\
Furthermore, when the signature and rank of the given lattice is clear, we will omit $\romantwo_{n_+,n_-}$ from the genus symbol.\\
The relation of the genus symbols of a lattice and the respective negated lattice is noted in \parencite[Lemma 4.3]{ohashi2024finite}.
\begin{Lemma}
    Let $q$ be a prime and $L$ be an lattice whose $q$-adic symbol is of the form 
    \[
    1_{(t_0)}^{\epsilon_0n_0}q_{(t_1)}^{\epsilon_1n_1}\cdots \left(q^e \right)_{(t_e)}^{\epsilon_en_e}
    \]
    where $t_i \in \mathbb Z/8\mathbb Z \cup \{ \romantwo\}$ appears only when $q=2$. Then a $q$-adic symbol of the negated lattice $L[-1]$ is as follows
    \[
    1_{(t_0')}^{\epsilon_0'n_0}q_{(t_1')}^{\epsilon_1'n_1}\cdots \left(q^e \right)_{(t_e')}^{\epsilon_e'n_e}
    \]
    where $\epsilon_i' = \left( \frac{-1}{q}\right)^{n_i}\epsilon_i$ $(q\geq 3)$ and $\epsilon_i' = \epsilon_i $, $t_i' = -t_i$ $(t_i \in \mZ/8\mZ)$, $t_i' = \romantwo$ $(t_i = \romantwo)$ $(q=2)$.
\end{Lemma}
\subsubsection{Existence of lattices}
\noindent In this work it is necessary to ask for the existence of an even lattice $L$ of a given rank, signature and such that $q_L$ is in a specific isometry class of torsion quadratic forms. \\
Let $L$ be an even lattice with signature $(t_{(+)},t_{(-)})$ and $q_\discgr{L}$ the quadratic form on the discriminant group of $L$. 
We define the \textit{signature mod 8} as
\[
\sign\ q_{L} \coloneqq t_{(+)} -t_{(-)} \mod 8.
\]
The signature mod 8 of the quadratic form on the discriminant group is well-defined, i.e.\ a different lattice with isomorphic discriminant groups will produce the same signature \parencite[Theorem 1.3.1, 1.3.3]{Nikulin_1980}.
We briefly recall an important result on the existence of certain lattices from the work of Nikulin \parencite{Nikulin_1980}.
\begin{Theorem}\parencite[Theorem 1.10.1]{Nikulin_1980}\label{NikulinExistenceResult} An even lattice with invariants $(t_{(+)},t_{(-)},q)$ exists if and only if the following conditions are satisfied:
\begin{itemize}
    \item[1)] $t_{(+)}-t_{(-)} \equiv \sign\ q\mod 8$
    \item[2)] $t_{(+)}\geq 0,t_{(-)}\geq 0$, $t_{(+)}+t_{(-)} \geq l(A_q)$
    \item[3)] $(-1)^{t_{(-)}} \vert A_q \vert \equiv \det  K(q_p) \mod {(\mathbb Z_p^*)^2}$ for all odd primes with $t_{(+)}+t_{(-)}=l_p(A_q)$.
    \item[4)] $\vert A_q \vert \equiv \pm \det K(q_2) \mod {(\mathbb Z_2^*)^2}$ with $t_{(+)}+t_{(-)}=l_p(A_2)$ and $q_2\neq q_\theta^{(2)}(2) \oplus q'_2$.
\end{itemize}
\end{Theorem}
\begin{Remark}
    The term $K(q_p)$ denotes a $\mathbb Z_p$ lattice that has the quadratic form $q_p$ as its discriminant group. The group $A_q$ is the discriminant group associated to the quadratic form $q$. The original statement in \parencite{Nikulin_1980} uses the notation $\text{discr}(L)$ for the determinant.
\end{Remark}
\subsection{\texorpdfstring{$K3$ surfaces}{K3 surfaces}}\label{K3s}
Let $X$ be a $K3$ surface over an algebraically closed field $k$.
If the rank of the Néron-Severi group $\textrm{NS}(X)$ is 22, which is equal to the second Betti number, $X$ is called \textit{supersingular}. Supersingular $K3$ surfaces only occur in positive characteristic. From now on, assume that $X$ is supersingular.\\
Endowed with the intersection product, we can consider $\NS$ as an even lattice. By the Hodge index theorem, $\NS$ has signature (1,21). The discriminant group is 
\[
\discg \cong\left( \faktor{\mathbb Z}{p \mathbb Z}\right)^{2\sigma}, 
\]
where $\sigma\in \{1,2,\dots,10  \}$ as is shown in \parencite{ArtinSupSing}.
The positive integer $\sigma$ is called the \textit{Artin invariant} of $X$. One can show that in the case where $\sigma = 1$, the quadratic form on $A_\NS$ is \textit{anisotropic}, i.e.\ that nontrivial elements $x\in A_\NS$ have $q_\NS(x)\neq 0$, by reconstructing the quadratic form on the discriminant group from the genus symbol. \\ 
For a fixed characteristic $p$, the supersingular $K3$ surface of Artin invariant 1 is unique up to isomorphism (\parencite{OgusK3Cryst} for $p\geq 3$, \parencite{Char2Dolgachev} and \parencite{RudakovShafarevich} for $p=2$). The supersingular $K3$ surface of Artin invariant 1 is sometimes called the \textit{superspecial} $K3$ surface. 
The genus of $\NS$ for the superspecial $K3$ surface is given by
\begin{align*}
    &\textrm{II}_{1,21}p^{\epsilon2},  \text{ for } p\geq 3 \text{ and with } \epsilon = -\left ( \frac{-1}{p}\right)\\
    &\textrm{II}_{1,21}2^{-2}_{\romantwo},  \text{ for } p = 2.
\end{align*}
The definition of $\epsilon$ uses the Legendre symbol.\\
An automorphism $g$ of $X$ acts on $\NS$ via the pullback of line bundles.
If the induced action on $H^0(X,\Omega^2_X)$ is the identity, $g$ is called \textit{symplectic}. Otherwise, $g$ is \textit{non-symplectic}.
We say that $g$ is \textit{wild} if $p \mid \text{ord}(g)$ and \textit{tame} if $\text{ord}(g)$ is coprime to $p$. A wild automorphism whose order is a prime power is always symplectic. Wild automorphisms only appear in characteristic $\leq 11$ \parencite[Theorem 2.1]{DolgachevKeum2009}. We say that a finite group $G\subset \Aut(X)$ is wild if $p\mid\vert G\vert $, i.e.\ if $G$ contains a wild automorphism.\\ 
In \parencite{Ogus1983}, an analog of the Torelli-type theorem for complex $K3$ surfaces is introduced for supersingular $K3$ surfaces in characteristic $p\geq5$. For our purpose, we follow the introduction to characteristic subspaces from \parencite{JangIndex} while omitting some details.\\
The discriminant group $\discg$ is an $\mathbb F_p$-vector space of dimension $2\sigma$. We extend the Frobenius endomorphism $F_k$ to $\discg$ by
\[
f \coloneqq \id \otimes F_k\colon \discg \otimes k \rightarrow \discg \otimes k.
\]
We say that a $\sigma$-dimensional totally isotropic subspace $\mathcal K$ of $\discg \otimes k$ is a \textit{characteristic subspace} if $\mathcal{K} + f(\mathcal{K})$ is of dimension $\sigma +1$. A characteristic subspace is called \textit{strictly} characteristic if $\mathcal K \cap \discg = \{0\}$ in $\discg\otimes k$.  \\

For the superspecial $K3$ surface, $\discg$ contains exactly two strictly characteristic subspaces $C_1$ and $C_2$, such that $f(C_1)=C_2$.
\begin{Theorem}[Crystalline Torelli Theorem for superspecial $K3$ surfaces,\parencite{Ogus1983}]\label{crystallinetorelli}
    Let $X$ be the superspecial $K3$ surface in characteristic $p\geq 5$. 
 Then an isometry $g\in O(\NS)$ is induced by a (unique) automorphism $\alpha \in \Aut(X)$ if and only if $g\otimes \mathbb R$ preserves the ample cone and $g_A\otimes k$ preserves both strictly characteristic subspaces.
\end{Theorem}
\noindent The crystalline Torelli Theorem gives an explicit way to think about automorphisms in a lattice-theoretic way. The property that an automorphism is symplectic can also be detected in this setting.
\begin{Theorem}\parencite[Theorem 3.1]{ohashi2024finite}\label{critsymp}
    Let $X$ be a supersingular $K3$ surface. Then an automorphism $g$ is symplectic if and only if the induced action on $A_\NS$ is trivial.
\end{Theorem}
\noindent Let $G\subset \text{Aut}(X)$ be a finite group acting faithfully on $X$ and let $G_s$ be the normal subgroup of symplectic automorphisms. Then Theorem \ref{critsymp} allows us to interpret $G/G_s$ as a subgroup of the automorphism group of $A_\NS$ (i.e.\ group isomorphisms that respect the induced quadratic form).

\section{\texorpdfstring{Actions on supersingular $K3$ surfaces}{Actions on supersingular K3 surfaces}}

\subsection{Known results in positive characteristic}

Let $X$ be a supersingular $K$3 surface of Artin invariant $\sigma$ in characteristic $p>0$.
The following result by Nygaard gave a condition on the non-symplectic index of the full automorphism group.
\begin{Theorem}[\parencite{Nygaard}]\label{NonSympIndex}
    Let $X$ be a supersingular $K3$ surface of Artin invariant $\sigma$ and $\Aut(X)$ the automorphism group of $X$. Then the non-symplectic index $[\Aut(X):\Aut_s(X)]$ divides $p^\sigma +1$.
\end{Theorem}
This result was further improved by Jang in \parencite{JangIndex} for $p\geq 5$ by giving a description of families in the moduli space of supersingular $K3$ surfaces with Artin invariant $\sigma$ such with a specific non-symplectic index. The result most important to our work is that there exists a (unique) supersingular $K3$ surface of Artin invariant $\sigma$, such that the non-symplectic index is equal to $p^\sigma+1$. Since the superspecial $K3$ surface is unique, we can conclude that the non-symplectic index of its automorphism group is $p+1$.\\
Now let $X$ be the superspecial $K3$ surface, $G\subset\Aut(X)$ be a finite group and $G_s\subset G$ be the maximal subgroup, such that every element of $G_s$ acts symplectically on $X$. Then $G$ is an \textit{extension} of the symplectic automorphism groups $G_s$. It is known that $G_s$ is a normal subgroup and, by Theorem \ref{NonSympIndex}, finite. Furthermore, since the symplectic automorphisms are the kernel of the map
\[
\Aut(X)\rightarrow\GL(H^0(X,\Omega_{X}^2))\cong k^\times
\]
we may conclude that the quotient is cyclic. We note further conditions for the non-symplectic index of a finite group $G$. 
\begin{Proposition}\parencite[Prop. 4.1]{KEUM201639}
     Let $X$ be the superspecial $K3$ surface in characteristic $p\geq 5$. Consider a finite group $G\subset \textrm{Aut}(X)$. Then the non-symplectic index $n=\vert G/G_s\vert$ fulfills $\varphi(n) \leq 20 $ and $n \mid p+1 $. Here, $\varphi$ denotes Euler's totient function.
\end{Proposition}

\noindent Let $G\subset O(\NS)$ be a finite group. An important object in the study of actions on lattices is the \textit {invariant lattice} (or \textit{fixed point lattice} )
$$\NS^G\coloneqq \left \{ x\in \NS \ \vert \ g(x)=x  \right\}.$$
One can show that $\NS^G$ is nondegenerate.

\noindent We define the \textit{co-invariant lattice} as the orthogonal complement of the invariant lattice
\[
\NS_G \coloneqq (\NS^G)^\perp .
\]

\begin{Remark}\label{eigenvalues} Let $G_s\subset \Aut_s(X)$.
An extension $G$ of $G_s$ must also preserve the decomposition $\NS^{G_s} \oplus \NS_{G_s}$. This can be seen from the fact that since $G_s$ is a normal subgroup, for every $g\in G$, \[g^{-1}G_sg \subset G_s.\]
Note, however, that although an extension $G$ must map $\NS^{G_s}$ to itself, it need not induce the identity on $\NS^{G_s}$.
\end{Remark}
The following Lemma consists of important, well-known results.
\begin{Lemma}\label{usefulProperties}
Let $X$ be the superspecial $K3$ surface.
\begin{itemize}
    \item[a)] There is an embedding $\text{Aut}(X)\hookrightarrow O(\NS)$. 
    \item[b)] Additionally an isometry $g \in O(\NS)$ that is induced by an automorphism is of finite order if and only if it fixes an ample class.
    \item[c)] A group $G \subset O(\NS)$ that is induced by automorphisms is finite if and only if $\NS^G$ contains an ample class.
    \item[d)] For subgroups $H\subset G\subset O(\NS)$, we know that $\NS^G \subset \NS^{H}$.
\end{itemize}
     
\end{Lemma}
\begin{proof}
    The statement a) is shown in \parencite[Section 8]{Rudakov1983SurfacesOT}.\\
    For statement b) consider an ample class $H\in \NS$. Then for an isometry $g \in O(\NS)$ of finite order $n$, the class $\Tilde H=\sum_{i=1}^ng^i(H)$ is fixed by $g$. If $g$ is induced by an automorphism, $\Tilde{H}$ is non-zero since $g$ preserves the ample cone of $X$, so $g^i(H)$ is again contained in the ample cone and $\Tilde H$ is ample.    
    Conversely, assume that an isometry $g$ fixes an ample class. By a) $g$ acts faithfully on the orthogonal complement $H^\perp$, which has signature (0,21) since every ample divisor has positive self-intersection by the Nakai-Moishezon criterion for ampleness. A negative definite lattice always has a finite automorphism group, so $g$ is of finite order.\\
    For statement c), we may adapt the argument for statement b) by considering the class $\sum_{g\in G}g(H)$.\\ 
    Statement d) follows from the definition of the invariant sublattice.
\end{proof}

\subsection{Classification of finite symplectic actions}\label{MatthiasandOhashi}
The finite groups acting faithfully and symplectically on the superspecial $K3$ surface have been classified independently in \parencite{wang2024finitegroupssymplecticautomorphisms} for characteristic $p\neq 2$ and \parencite{ohashi2024finite} for any characteristic. 
\begin{Corollary}\label{Leech}\parencite[Theorem 4.1]{ohashi2024finite}, \parencite[Lemma 3.7]{wang2024finitegroupssymplecticautomorphisms}
Let $X$ be the superspecial $K3$ surface in characteristic $p$ and $\Lambda$ be the Leech lattice, i.e.\ the negative definite unimodular lattice of rank 24 that contains no $-2$ classes and is unique up to isometry. For every finite subgroup $G_s \subset \Aut_s(X)$ there is a subgroup
of the Conway group $G'\subset \textsc{Co}_0 = O(\Lambda)$ such that $G_s\cong G'$ and the rank of $\Lambda^{G'}$ is at least 3. \\
Furthermore, $\Lambda_{G'}\cong \NS_{G_s}$.
\end{Corollary} 
\noindent In \parencite{HOHN2016618}, the invariant lattices of all actions on the Leech lattice have been classified. There is a difference in notation in \parencite{HOHN2016618}, i.e.\ the Leech lattice is defined to be positive definite, whereas we define it to be negative definite. In the list of Höhn and Mason, the group which is listed for each invariant lattice $\Lambda^{G_s}$ is the kernel of the map $O(\Lambda_{G_s})\rightarrow O(q_{\Lambda_{G_s}})$.

\begin{Theorem}[\parencite{ohashi2024finite}]\label{matthiaswild} The maximal groups acting faithfully and symplectically on the superspecial $K3$ surface in characteristic $p\leq 11$ are
\begin{itemize}
    \item[\textsc{(}$p=2$\textsc{)}] $M_{21}.2$ and
        $\text{Aut}(\mathfrak S_6)$
    \item[\textsc{(}$p=3$\textsc{)}] $\textsc{PSU}(4,\mathbb F_{3^2})$, $2^2.\mathfrak A_{4,4}$ and $2^4:\mathfrak{S}_{3,3}$
    \item[\textsc{(}$p=5$\textsc{)}] $\textsc{PSU}(3,\mathbb F_{5^2})$, $\mathfrak A_8$ and $2^4:(3\times \mathfrak A_5):2$
    \item[\textsc{(}$p=7$\textsc{)}] $M_{21}.2$ and $2^4:\mathfrak A_7$
    \item[\textsc{(}$p=11$\textsc{)}] $M_{22}$, $M_{11}$ and $2^4:\mathfrak S_{3,3}$
\end{itemize}
    
\end{Theorem}
\begin{Theorem}[\parencite{ohashi2024finite}]\label{matthiasmainresulttame}
A finite group $G_s$ admits a tame symplectic action on the superspecial $K3$ surface $X$ in characteristic $p$ if and only if $p \nmid |G_s|$ and $G_s$ can be realized as a subgroup of $M_{23}$ whose action on $\{1, 2,\dots, 24\}$ has
\begin{itemize}
    \item[i)] either at least 5 orbits
    \item[ii)] or exactly 4 orbits such that the orbit lengths $l_1,\dots,l_4$ satisfy the condition
    \begin{equation*}\label{conditionTameGroup}
        \left ( \frac{l_1\cdots l_4}{p}\right)=-1.
    \end{equation*}
\end{itemize}
\end{Theorem}
\noindent
Let $X$ be the superspecial $K3$ surface in characteristic $p$ and $G_s\subset\Aut_s(X)$ be a finite group. We say that $\Tilde G_s$ is the \textit{saturation} of $G_s$ if $\Tilde G_s$ is the largest subgroup of $\Aut_s(X)$ such that $\NS^{G_s}\cong \NS^{\Tilde G_s}$. We say that $G_s$ is \textit{saturated} if $G_s = \Tilde G_s$. Clearly, every maximal symplectic group is saturated.  

\begin{Lemma}\label{asfixedpointlattice}
    Let $X$ be the superspecial $K3$ surface in characteristic $p>11$. Then the maximal tame groups in a given characteristic either have exactly four orbits and fulfill the condition (ii) in Theorem \ref{matthiasmainresulttame} or they are one of the maximal symplectic groups over the complex numbers classified by Mukai \parencite{Mukai1988}.\\
    For a  maximal symplectic tame group $G_s\subset \Aut_s(X)$ the invariant lattice $\NS^{G_s}$ has rank 2 or 3.
    If $G_s$ is a saturated tame group with an invariant lattice of rank 2, then $G_s$ is maximal.
\end{Lemma}
\begin{proof}
    As noted, maximal groups are saturated and thus appear in Table 1 of H\"ohn and Mason \parencite{HOHN2016618}. The work of Wang and Zheng \parencite[Table 1]{wang2024finitegroupssymplecticautomorphisms} includes a table of all saturated groups $G_s$ which embed into the symplectic automorphism group of a superspecial $K3$ surface, including the restrictions on the possible characteristic. Furthermore, the rank of the coinvariant lattice associated to the action of $G_s$ on the Néron-Severi lattice is also included.\\
    By considering this table, we may see that the groups with exactly four orbits are the groups with a coinvariant lattice of rank 20. This implies that they are maximal. \\
    The tame groups with rank 19 coinvariant lattice are the maximal groups over the complex numbers. We can also see that all tame groups with a coinvariant lattice of lower rank are contained in them, but they may not be maximal in the symplectic automorphism group due to the existence of finite subgroups which fulfill condition (ii) of Theorem \ref{matthiasmainresulttame}.
\end{proof}
\subsection{A criterion for the absence of extensions}

Let $X$ be the superspecial $K3$ surface in characteristic $p$ and $G\subset \Aut(X)$ a finite group. We know that $\NS^G$ contains an ample class $H$ (i.e.\ is non-empty) by Lemma \ref{usefulProperties} and must be a primitive sublattice in $\NS^{G_s}$ by definition.

\begin{Lemma}\label{AusschlussDuallyPrimitive}
Let $X$ be the superspecial $K3$ surface in characteristic $p\geq 5$ and $G\subset\Aut(X)$ be a finite group. Assume there is an $x\in \NS^{G}$ that is primitive and not dually primitive in $\NS$. Then $G/G_s$ is trivial.
\end{Lemma}
\begin{proof}
Let $\Tilde{x}$ be a non-trivial class in $\discg$ that is represented by a rational multiple of $x$.\\
Since $\discg \cong (\mathbb Z/p\mathbb Z )^2$ as a group, we know that the automorphism group of $\discg$ is a subgroup of the automorphism group of $(\mathbb Z/p\mathbb Z )^2$. This inclusion is proper, since automorphisms of $\discg$ have to preserve the quadratic form $q_\NS$. 
As a consequence, $G/G_s$ can be realized as a proper subgroup of $\textrm{Aut}\left( (\mathbb Z / p\mathbb Z)^2 \right)$. 
Since $x\in \NS^G$, the induced action of $G/G_s$ must fix $\Tilde{x}$ in \discg. \\
The orthogonal complement of $\Tilde{x}$ in $\discg$ is preserved too, which is isomorphic to $\mathbb Z/p\mathbb Z $. We recall that the quadratic form on $\discg$ is anisotropic, so $\Tilde x$ is not in the orthogonal complement of $\Tilde{x}$. Consequently, the quotient $G/G_s$ fixes $\Tilde x$ and maps a generator of the orthogonal complement to a multiple of itself.\\ 
Hence, $G/G_s$ must also be contained in $\text{Aut}(\mathbb Z/p\mathbb Z )\cong(\mathbb Z / p\mathbb Z )^*$. But this group has order $p-1$, so by Theorem \ref{NonSympIndex} the order of $G/G_s$ can only be 1 or 2. 
If the order is 2, the action is a reflection along a subspace of $\mathbb F_p^2$. But such a reflection cannot preserve the strictly characteristic subspace of $X$, since the intersection of the strictly characteristic subspace with $\discg$ is trivial.
Thus, by the crystalline Torelli theorem, it is not induced by an automorphism. As a consequence the non-symplectic index must be 1.
\end{proof}
\begin{Remark}\label{explanationduallyprim}
Let $X$ be the superspecial $K3$ surface in characteristic $p\geq 5$ and $G_s\subset \Aut_s(X)$ a finite group.
    Lemma \ref{AusschlussDuallyPrimitive} implies that the ample divisor which is fixed by a non-trivial extension $G$ of $G_s$ (i.e. $[G:G_s]>1$) has to be dually primitive in $\NS$. More importantly, if no primitive element of $\NS^{G_s}$ is dually primitive in $\NS$, Lemma \ref{usefulProperties} d) and Lemma \ref{AusschlussDuallyPrimitive} imply that there is no non-trivial extension of $G_s$ (i.e. $G = G_s$ for any extension of $G_s$). \\
    This fact is frequently used in Sections \ref{tame} and \ref{wild}.
\end{Remark}
\begin{Example}\label{RescalingExample}
    Let $L\subset\NS$ be a rank 2 primitive sublattice. Assume that $p\nmid \det(L^\perp)$. Then $L=L'[p]$ for a lattice $L'$  since the $p$-length of $A_L$ is necessarily 2. Since no gluing along $p$ occurs when taking $\NS$ as a primitive extension of $L$ and $L'$, no element of $L$ is dually primitive in $\NS$.
\end{Example}

\section{Extensions of maximal tame group actions}\label{tame}
\noindent Let $p> 11$. Then we know by \parencite[Theorem 2.1]{DolgachevKeum2009} that all maximal groups acting faithfully and symplectically on the superspecial $K3$ in characteristic $p$ are tame. Note that in characteristic 11, there is also the maximal group $2^4:\mathfrak S_{3,3}$ acting tamely (see \parencite{ohashi2024finite}). All other maximal groups in characteristic $\leq 11$ are wild and will be discussed in section \ref{wild}. \\In this section we prove Theorem \ref{TameCase}.

\noindent By Lemma \ref{asfixedpointlattice}, we know that maximal symplectic groups are precisely the ones with rank 2 or rank 3 invariant lattice. We will discuss the groups with rank 2 invariant lattice (i.e.\ ones which do not appear over the complex numbers) in subsection \ref{fourorbits}.\\
The groups with rank 3 invariant lattice are the maximal groups over the complex numbers and are discussed in Subsection \ref{MukaiGroups}. These groups are not maximal in every characteristic $p\geq 11$, but Theorem \ref{TameCase}(ii) holds independently of their action being maximal or not.

\subsection{Groups which act with exactly four orbits}\label{fourorbits}
Groups with exactly four orbits that act in positive characteristic do not act faithfully on any $K3$ surface over the complex numbers. The maximal such groups are given in Table 3 below (see \parencite[Table 4]{ohashi2024finite}):
\begin{center}
{\rowcolors{2}{white}{lightgray}
\begin{tabular}{ |c|c|c| } 
\hline
symplectic automorphism group $G_s$ & group order & genus symbol of $\Lambda_{G_s}$  \\
\hline
$M_{21}$ & $2^7.3^2.5.7$  &$2_{\text{II}}^{-2} 3^{-1}7^{-1}$ \\ 
$2^4:\mathfrak A_6$& $2^7.3^2.5$ &$4^{-1}_5 8_1^{+1}3^{+1}$  \\ 
$\mathfrak A_7$& $2^3.3^2.5.7$  & $3^{+1} 5^{+1} 7^{+1} $ \\ 
$M_{20}:2$ & $2^7.3.5$ & $4^{-1}_3 8_1^{+1}5^{-1}$  \\
$2^3:L_2(7)$ & $2^6.3.7$ & $4_2^{+2}7^{+1}$\\
$ 2^2.\mathfrak A_{4,4} $ & $ 2^7.3^2 $ & $ 8_6^{-2}3^{-1} $ \\
$ \mathfrak S_6 $ & $ 2^4.3^2.5 $ & $2_{\text{II}}^{-2}3^{+2}5^{+1}  $ \\
$M_{10}$ &$2^4.3^2.5$&  $2^{+1}_5 4^{+1}_13^{-1}5^{+1}   $\\
$2^4:\mathfrak S_{3,3}$ & $2^6.3^2 $   &$4^{+1}_78^{+1}_13^{+2}  $ \\
$3^2:QD_{16} $ & $2^4.3^2 $ & $2^{+1}_14^{+1}_13^{-1}9^{-1} $\\
\hline
\end{tabular}}
\captionof{table}{Maximal tame symplectic automorphism groups with exactly four orbits}\label{Rk2tame}
\end{center}

\noindent These groups act on the superspecial $K3$ surface $X$ in characteristic $p$ if and only if there exists a rank 2 indefinite lattice with which we can glue $\Lambda_{G_s}$ into \NS. Here, $\Lambda$ is the Leech lattice. Note that $\Lambda_{G_s}$ does not glue into the Néron-Severi lattice for every $p$, as is seen in Theorem \ref{matthiasmainresulttame}(ii).
\begin{Proposition}\label{rk2222}
    Let $G_s$ be a group from Table \ref{Rk2tame} and $p$ be a prime such that $G_s$ acts tamely on the superspecial $K3$ surface in characteristic $p$. Then $G_s$ has no non-trivial extension.
\end{Proposition}
\begin{proof}
    By the assumption that $G_s$ acts tamely, we can see that $p\nmid \det(\Lambda_{G_s})$ by Table \ref{Rk2tame}. Since $\text{rk }\NS^{G_s} = 2$ and $l_p(A_{\NS^{G_s}} ) = 2$, we are in the same sitation as Example \ref{RescalingExample}. 
     Consequently, by Lemma \ref{AusschlussDuallyPrimitive} any extension of $G_s$ is trivial.
\end{proof}
\begin{Remark}
    This proves part (i) of Theorem 1.2.
\end{Remark}
\subsection{Mukai groups}\label{MukaiGroups}
Let $X$ be the superspecial $K3$ surface in characteristic $p\geq 5$.
We investigate the possible extensions of the maximal groups that act symplectically over the complex numbers, as classified by Mukai in \parencite{Mukai1988}. We call these groups the \textit{Mukai groups}. These are groups $G_s$ such that $\NS^{G_s}$ has rank 3. Extensions of these groups over $\mathbb C$ have been classified in \parencite{brandhorst-hashimoto}. \\
 The Mukai groups act faithfully and symplectically on the superspecial $K3$ surface in every characteristic $p\geq 3$ \parencite[Proposition 10.1]{ohashi2024finite}. This action can be maximal in characteristic $p> 11$, as is noted in \parencite[Theorem 1.6]{ohashi2024finite}.  Our argument will work with the assumption that the characteristic does not divide the determinant of the coinvariant lattice and is greater or equal to 5. This holds true for any prime $\geq 11$, but in some cases also allows us to discuss possible extensions in characteristic 5 and 7, even if the action is wild. Furthermore, the Theorem holds without assuming that the action of $G_s$ on $X$ is maximal. We will show the following:
 \begin{Theorem} \label{Mukaimainresult}
     Let $X$ be the superspecial $K3$ surface in characteristic $p\geq 5$. Let $G_s$ be a Mukai group such that $p\nmid \det(\NS_{G_s})$. Then $G_s$ has an extension with non-symplectic index given in the table below     
     \begin{center}
     \small
     {\rowcolors{2}{white}{lightgray}
\begin{tabular}{ |c|c|c|c|c|c| } 
\hline
$G_s$&$\vert\det(\Lambda_{G_s})\vert$& $\vert G :G_s\vert =1$&$\vert G :G_s\vert =4$ & $\vert G :G_s\vert =6$ & $\vert G : G_s \vert = 2$\\
\hline
$T_{48}$ & $2^7.3$ &  $p  \equiv 1 \mod 24$ & - & $p\equiv 5 \mod 6$ & otherwise\\
$N_{72}$ & $2^2.3^4$ & $p  \equiv 1 \mod 24$ & $p\equiv 3 \mod 4$ & - & otherwise\\
$M_9$ & $2^3.3^3$ &$p  \equiv 1 \mod 24$ & - & $p\equiv 5 \mod 6$& otherwise\\
$\mathfrak S_5$ &$2^2.3.5^2$& $p \equiv 1,49 \mod 120$ &- &-& otherwise\\
$L_2(7)$ & $2^2.7^2$ & $p  \equiv 1,9,25 \mod 56$& $p\equiv 3 \mod 4$ & -& otherwise\\
$H_{192}$ & $2^7.3$& $p  \equiv 1 \mod 24$ & -&-& otherwise\\
$T_{192}$ & $2^6.3$& $p  \equiv 1 \mod 24$ & - & $p\equiv 5 \mod 6$& otherwise\\
$\mathfrak A_{4,4}$ & $2^5.3$& $p  \equiv 1 \mod 24$& $p\equiv 3 \mod 4$ & -& otherwise\\
$\mathfrak A_{6}$ &$2^2.3^2.5$ &  $p
\equiv 1,49 \mod 120$ & $p\equiv 3 \mod 4$ & -& otherwise\\
$F_{384}$ & $2^8$ & $p\equiv 1 \mod 8$& $p\equiv 3 \mod 4$ & -& otherwise\\
$M_{20} $& $2^5.5$ & $p  \equiv1,9 \mod 40$&$p\equiv 3 \mod 4$ & -& otherwise\\
\hline
\end{tabular}}
\normalsize
\end{center}
Any other extension of $G_s$ must have a non-symplectic index which divides the one above.

 \end{Theorem}

\subsubsection{Comparison to the complex case}
We briefly discuss the methods used in \parencite{brandhorst-hashimoto}. For general information on complex $K3$ surfaces we refer to \parencite{Huybrechts_2016}.

\noindent When considering a complex projective $K3$ surface $Y$, we have a finite index decomposition:
\[
\NSY \oplus T(Y)  \subset H^2(Y,\mathbb Z).
\]
Here, $H^2(Y,\mathbb Z)\cong E_8^{\oplus 2}\oplus U^{\oplus3}$ is the \textit{K3 lattice} of signature (3,19) and the \textit{transcendental lattice} $T(Y)$ is the orthogonal complement of $\NSY$ in $H^2(Y,\mathbb Z)$. We note that $H^2(Y,\mathbb Z)$ is even and unimodular. All actions by the maximal symplectic groups over the complex numbers are obtained on \textit{singular} $K3$ surfaces, i.e.\ complex $K3$ surfaces with Picard rank 20 and $T(Y)$ of rank 2. This is the highest possible Picard rank over the complex numbers. The second cohomology of a $K3$ surface is equipped with a \textit{Hodge structure}. This is a decomposition of the form
\[
H^2(Y,\mZ)\otimes \mC \cong H^2(Y,\mC) = H^{2,0}(Y)\oplus H^{1,1}(Y) \oplus  H^{0,2}(Y).
\]
For a singular $K3$ surface $Y$, we have that $H^{1,1}(Y)=\NSY\otimes \mC$ and $H^{2,0}(Y)$ and $H^{0,2}(Y)$ respectively are the two linearly independent isotropic sub-vector spaces in $T(Y)\otimes \mC$. 
For more general details we refer to \parencite[Section 3]{brandhorst-hashimoto}. The global Torelli theorem gives a lattice theoretic way of characterizing automorphisms.

\begin{Theorem}[global Torelli Theorem](see for instance \parencite[Theorem 5.3]{Huybrechts_2016})
    An isometry $g\in O(H^2(Y,\mZ))$ is induced by an automorphism if and only if $g\otimes \mathbb C$ preserves the Kähler cone and the Hodge structure. 
\end{Theorem}
\noindent Given a finite group $G$, the symplectic subgroup $G_s$ is a finite index normal subgroup with cyclic quotient $G/G_s$, as is shown in \parencite[Lemma 2.1]{Sterk1985}.
There is also a lattice theoretic way of recognizing wether or not an automorphism is symplectic or non-symplectic. Namely, an automorphism $\alpha$ is symplectic if and only if the induced action of $\alpha$ on $T(Y)$ is trivial \parencite[Remark 15.1.2]{Huybrechts_2016}.\\
In particular, this implies that for any extension $G'$ of a symplectic group the finite, cyclic quotient group $G'/G'_s$ acts faithfully on $T(Y)$. The order of any finite, cyclic action on $T(Y)$ is bounded by $\varphi(n)$, where $n$ is the dimension of $T(Y)$. This allows us to find a natural bound on the non-symplectic index.
\begin{Lemma}\parencite[Lemma 5.1]{brandhorst-hashimoto}
     Let $G'_s$ be one of the maximal symplectic groups (i.e.\ the Mukai groups) and $Y$ be a complex $K3$ surface with $G'_s\subset \Aut_s(Y)$. Furthermore, let $G'\subset\Aut(Y)$ such that $G'_s$ is the symplectic subgroup of $G'$. Then non-symplectic index $n= \vert G' /G'_s\vert$ is bounded by $\varphi(n) \leq 2$, i.e.\ $n =1,2,3,4,6$.
\end{Lemma}
 \noindent To classify what extensions of the Mukai groups actually occur, Brandhorst and Hashimoto classify the possible transcendental lattices $T(Y)$ which glue into the invariant lattice $H^2(Y,\mZ)^{G_s}$ such that the automorphism group $\Aut(Y)$ contains a non-trivial extension of $G_s$. They show that there are only finitely many possibilities for finite extensions. To summarize, they compute the possible values $H^2$ of an ample primitive divisor in $H^2(Y,\mZ)^{G'_s}$ such that for the finite index decomposition
 \[
 H^2(Y,\mZ)^{G'_s} \supseteq \langle H\rangle \oplus \underbrace{\langle H \rangle^\perp}_{\coloneqq T(Y)}
 \]
there is an isometry of $H^2(Y,\mZ)^{G'_s}$ which acts trivially on $H$ and non-trivially on $T(Y)$. The list of possibilities can be found in (\parencite[Section 6]{brandhorst-hashimoto}).

\subsubsection{Extensions of Mukai groups}

\noindent
Let $Y$ be a projective complex $K3$ surface such that $G'_s\subset \Aut_s(Y) $ is a Mukai group.
As has been noted in the previous subsection, a finite group $G'\subset \Aut(Y)$ with symplectic subgroup $G'_s$ must induce an action on $T(Y)$ of order equal to the non-symplectic index of $G'$. Since we assume that $G'_s$ is a Mukai group, this index is equal to 1,2,3,4 or 6 by \parencite[Lemma 5.1]{brandhorst-hashimoto}, and $O(T(Y))$ must contain an element of this order.\\  
The possible orthogonal groups of rank 2 definite lattices have been classified. Most importantly, up to scaling there are only two classes of lattices that have an isometry of order 4 or 6.
\begin{Proposition}  \label{twolattices}\parencite[Lemma 3.6, Table 2]{Hulek_Schütt_2012}
    Consider the lattice $A_2[-a]$ associated to the Gram matrix 
    \[
    a \begin{pmatrix}
        2  & 1 \\ 1 & 2
    \end{pmatrix},
    \]
    for a non-zero integer $a$.
     The orthogonal group of $A_2[-a]$ is isomorphic to $D_6$ (the dihedral group of a hexagon). For this reason, we will denote $L_6[a]\coloneqq A_2[-a]$.\\
   Consider the lattice  $(A_1^{\oplus 2})[-a]$ associated to the Gram matrix
    \[
     a \begin{pmatrix}
         2 & 0 \\ 0 & 2
     \end{pmatrix},
     \]
     for a non-zero integer $a$.
     The orthogonal group of $(A_1^{\oplus 2})[-a]$  is isomorphic to $D_4$ (the dihedral group of a square). We will denote $L_{4}[a]\coloneqq (A_1^{\oplus 2})[-a]$.\\
     Every even rank 2 definite lattice with an isometry of order 4 (resp. 6) is isomorphic to $L_{4}[a]$ (resp. $L_6[a]$) for some $a$. \\
     These are the only even rank 2 definite lattices with an orthogonal group that is not isomorphic to $\mathbb Z/2\mathbb Z$ or $\mathbb Z/2\mathbb Z\times \mathbb Z/2\mathbb Z$.
\end{Proposition}

\begin{Remark}\label{latticesnutzlich}
    \begin{itemize}
    \item Let $Y$ be a complex $K3$ surface. Let $G'\subset \Aut(Y)$ be a  finite group whose symplectic subgroup $G'_s$ is a Mukai group. The non-symplectic index $[G':G'_s]$ can only be equal to $4$ if $T(Y)\cong L_{4}[a]$ for a suitable integer $a$. The non-symplectic index $[G':G'_s]$ can only be equal to $3$ or $6$ if $T(Y)\cong L_{4}[a]$ for a suitable integer $a$.
        \item The lattice $L_{4}$ is the only positive definite lattice of rank 2 with determinant 4 and a discriminant group isomorphic to $(\mathbb Z/2\mathbb Z)^2$ as a group.
        \item The lattice $L_6$ is the only positive definite even lattice of rank 2 with determinant 3.
    \end{itemize} 
\end{Remark}

\begin{Theorem}\parencite[Theorem 5.1]{Hashimoto_2012}\label{Hashimotosurj}
    Let $G_s$ be one of the Mukai groups and $\Lambda$ be the Leech lattice. Then the natural map 
    \[
    O(\Lambda_{G_s})\longrightarrow O(q_{\Lambda_{G_s}})
    \]
    is surjective. 
\end{Theorem}
The coinvariant lattice $\Lambda_{G_s}$ associated to the Mukai group $G_s$ is isomorphic to the coinvariant lattice $\NS_{G_s}$ of the embedding $G_s\hookrightarrow \Aut(X)$. We recall this fact in Theorem \ref{Leech}. 
\begin{Corollary}
    \label{hashimotosurjective}
    Let $Y$ be a complex $K3$ surface and $G'_s\subset \Aut_S(Y)$ a Mukai group.
    Then every isometry of $\hyz^{G'_s}$ can be extended to an isometry of $\hyz$.
\end{Corollary}
\begin{proof}
    This follows from Theorem \ref{Hashimotosurj} and \parencite[Proposition 5.3]{Hashimoto_2012}.
\end{proof}

\begin{Proposition}\parencite[Section 6]{brandhorst-hashimoto}\label{simonhashimotoresult}
    Let $Y$ be a complex $K3$ surface with $G'_s\subset \Aut_s(Y)$ a Mukai group. Let $\langle H\rangle =(T(Y))^\perp$ in $\hyz^{G'_s}$ and let $m$ be the index $[\hyz^{G'_s}:\langle H \rangle \oplus T(Y)]$.
    The results of Brandhorst and Hashimoto show the following:
    \begin{itemize}
        \item[1)] If $m=1$ and $T(Y)\cong L_6[a]$ for a suitable $a$, then there exists an index 6 extension of $G'_s$ on $Y$.
        \item[2)] If $m=1,2$ and $T(Y)\cong L_{4}[a]$ for a suitable $a$, then there exists an index 4 extension of $G'_s$ on $Y$.
        \item[3)] If $m=1,2$, then there exists an index 2 extension of $G'_s$ on $Y$. 
        \item[4)] There is no non-trivial extension of the action of $G'_s$ on $Y$ if $m\geq 3$.
    \end{itemize}
\end{Proposition}

\subsubsection{The proof of Theorem \ref{Mukaimainresult}}
 
\noindent Let $X$ be the supersingular $K3$ surface in characteristic $p\geq 5$ and $G_s\subset \Aut_s(X)$ be one of the Mukai groups. Assume that $p \nmid \det(\NS_{G_s})$. We recall that $\NS_{G_s}$ is isomorphic to the coinvariant lattice of the action on the Leech lattice (\ref{Leech}). This also uniquely determines the genus of $\NS^{G_s}$.\\
We employ similar techniques as the complex case. Namely, we search for possible values of $H^2$ for an ample primitive divisor $H$ in $\NS^{G_s}$ such that there is an extension of $G_s$ which fixes $H$.
Since the invariant lattice of $G_s$ has rank 3 by Lemma \ref{asfixedpointlattice}, for any such $H$ there is a finite index decomposition
\begin{equation}\label{decompo}
    \NS^{G_s} \supseteq \langle H\rangle \oplus \underbrace{\langle H \rangle ^\perp}_{\coloneqq T'},
\end{equation}
where $T'$ is a rank 2 negative definite lattice. For any given extension $G\subset \Aut(X)$ of $G_s$, we know there is an ample divisor $H$ such that $G$ fixes $H$. Hence, for a suitable $H$ an extension $G$  of $G_s$ has to preserve the summands of the direct sum above. Our aim is to show that there is a relation between extensions of $G_s$ which act on the superspecial $K3$ surface in characteristic $p$ and extensions of $G_s$ over the complex numbers.
\begin{Theorem}\label{CorrespondenceComplexSupSpe}
    Let $X$ be the superspecial $K3$ surface in characteristic $p$ and $G_s\subset \Aut_s(X)$ be a Mukai group such that $p \nmid \det(\NS_{G_s})$. 
    Then there is a finite group $G\subset \Aut(X)$ with symplectic subgroup $G_s$ if and only if there is a complex $K3$ surface $Y$ such that 
    \begin{equation}\label{conditionsupsingred}
        \left ( \frac{\det(T(Y))}{p}\right) = -\left ( \frac{-1}{p}\right)
    \end{equation}
    and a finite group $G'\subset\Aut(Y)$ with $G'_s\cong G_s$ and $[G:G_s]=[G':G_s']$.
 
    \[
\begin{tikzcd}
\Aut(X)             & G \arrow[l,phantom, sloped, "\supset"]                       & G' \arrow[r,phantom, sloped, "\subset"]             & \Aut(Y)             \\
\Aut_s(X) \arrow[u,phantom, sloped, "\subset"] & G_s \arrow[l,phantom, sloped, "\supset"] \arrow[u,phantom, sloped, "\subset"] \arrow[r,phantom, sloped, "\cong"] & G'_s \arrow[r,phantom, sloped, "\subset"] \arrow[u,phantom, sloped, "\subset"] & \Aut_s(Y) \arrow[u,phantom, sloped, "\subset"]
\end{tikzcd}
    \]
    
\end{Theorem}
\noindent This result allows us to reduce our problem to a finite computation. In detail, it suffices to calculate condition (\ref{conditionsupsingred}) for every of the finitely many $K3$ surfaces in \parencite[Section 6]{brandhorst-hashimoto}. 
\noindent To prove Theorem \ref{CorrespondenceComplexSupSpe}, we will try to relate the decompositions of the invariant lattice in the complex case and in the supersingular case.\\
The following result allows us to bound the index $[ \NS^ {G_s} :\langle H\rangle \oplus T']  $.
\begin{Proposition}\parencite[Prop. 4.7]{Brandhorst_Hofmann_2023}\label{simontommy}
    Let $L$ be an integral lattice and $f \in O(L)$ an isometry of prime order $q$. The characteristic polynomial of $f$ is of the form $\chi_f = \Phi_1^{e_1}\Phi_q^{e_q}$, where $\Phi_i$ is the $i$-th cyclotomic polynomial. Set $B = \ker(\Phi_1(f))$, $C = \ker(\Phi_q(f))$ and $m = \min\{e_1,e_q,l(\discgr{B}),l(\discgr{C})  \}$. Then $qL \subseteq B \oplus C \subseteq L$ and $[ L : B \oplus C] \ \vert \ q^m$.
\end{Proposition}

\noindent In our case, we can establish some restrictions on the decomposition (\ref{decompo}).
\begin{Lemma}\label{einschränkungIndex}
    
    Let $X$ be the superspecial $K3$ surface in characteristic $p$ and let $G\subset\Aut(X)$ be a finite subgroup with symplectic group $G_s\neq G$. Assume that $G_s$ is a Mukai group with $p \nmid \det(\NS_{G_s})$ and $G = \langle g,G_s\rangle$.\\    
    Set $L = \NS^{G_s}$ and $i$ to be a natural number such that $g^i$ has prime order $q$ and is non-symplectic.
    Using the notation of Proposition \ref{simontommy} with $f=g^i$, both of the following statements hold
    \begin{itemize}
        \item[1)] $B = \langle H \rangle$ for an ample divisor $H$ which is primitive in $\NS$ and $p \nmid H^2  $. In particular $e_1=m=1$.\\
        Furthermore, $p \nmid [\NS: \langle H\rangle \oplus C\oplus \NS_{G_s} ]$ (i.e.\ no gluing along $p$ occurs).
        \item[2)] The non-symplectic index of $G$ is $n \in \{2,3,4,6\}$. Writing $C = T'$, $n$ depends on the index of the decomposition $[L:B\oplus C]$, i.e.\
        \begin{center}{\rowcolors{2}{white}{lightgray}
            \begin{tabular}{|c|c|}
            \hline
                 $[\NS^{G_s}: \langle H \rangle \oplus T']$& possible values for $n$  \\ \hline
                 1& 2,3,4,6\\
                 2 & 2,4\\
                 3 & 3\\
                 \hline
            \end{tabular}}
        \end{center}
        
    \end{itemize}
\end{Lemma}
\begin{proof}
    We may assume that $p\neq q$ and $\text{ord}(g)$ is not a power of $p$, since the automorphism would be wild and symplectic otherwise, which contradicts the assumption that $g$ is non-symplectic.\\
  \indent 1) We know that $B$ must contain an ample divisor $H$ that is primitive in $\NS$. This implies that $e_1$ is at least 1. We have taken $f$ not to be trivial, so $e_1$ is at most 2. 
  By assumption $p\nmid \det( \NS_{G_s})$ and thus the local genus symbol at $p$ of $\NS$ and $\NS^{G_s}$ must coincide.
  Furthermore, $p \neq q$ and thus $p \nmid [L :B\oplus C]$ by Proposition \ref{simontommy}. 
 
  \begin{Claim}
      There are suitable integers $b,c$ which are coprime to $p$ such that \\$\det(B)=b \text{ and } \det(C)= p^2 \cdot  c$ .
  \end{Claim}
  \begin{proof}[Proof of the Claim]
      Assume that the above claim is false. Then for suitable integers $b,c$ which are coprime to $p$, it must hold that $\det(B)= p \cdot b \text{ and } \det(C)= p \cdot c$ or $\det(B)= p^2 \cdot b \text{ and } \det(C)= c$ . Since $p \mid
  \det (B)$ and no gluing along $p$ occurs, there is a vector that is not dually primitive which would be fixed by $G$. Then Lemma \ref{AusschlussDuallyPrimitive} shows that $G$ is symplectic which contradicts our assumptions.
  \end{proof}
\noindent      The non-unimodular part of the local genus symbol at $p$ of $C$ must coincide with the non-unimodular part of the local genus symbol of $\NS$ at $p$. In other words, the $p$-Sylow groups of $A_C$ and $A_\NS$ equipped with the respective quadratic forms are isomorphic. Thus $l_p(A_C)=2$, which implies that $e_1=1$, $e_q=2$ and we may write $B = \langle H \rangle$. Additionally, $p \nmid [\NS: \langle H\rangle \oplus C\oplus \NS_{G_s} ]$.\\
  Since $l_p(A_C)=\rk C$, we may conclude that $C$ is the rescaling $C'[p]$ of an even, negative definite integral lattice $C'$.\\

  \indent 2) Set $T'=C$. 
  The proof of 1) has shown that $p \nmid [\NS: \langle H\rangle \oplus T'\oplus \NS_{G_s} ]$ and $T'= C'[p]$. The absence of gluings along $p$ implies that the elements of $T'$ are not dually primitive in $\NS$. Furthermore $G_s$ fixes $T'$ and thus the cyclic quotient $G/G_s$ has to act faithfully on $T'$.\\
  It follows that possible values of $n=[G:G_s]$ are $1,2,3,4,6$ since these are the only $n$-th roots of unity with cyclotomic polynomial of degree 2 or lower. From 1) we know that $m = 1$, so $[\NS^{G_s}: \langle H \rangle \oplus T'] \mid q$. We also note that $g$ has to fix $\langle H\rangle$ as well and thus also preserves the sublattices $\langle H \rangle $ and $ T'$ respectively.\\
  The list above follows directly in the case where $n$ is prime, and for $n =4$ by taking $f=g^2$. If $n=6$, by taking $g^ 2$ and $g^3$ it follows that $[\NS^{G_s}: \langle H \rangle \oplus T']$ must divide 2 and 3 and hence must be equal to 1.
\end{proof}

\begin{Remark}
    The proof of Lemma \ref{einschränkungIndex} also shows that $G/G_s$ is acting faithfully on $T'$. In particular, $G/G_s$ is a cyclic subgroup of $O(T')$.
\end{Remark}
\noindent We describe the aim of the following work in this section. Let $Y$ be a complex $K3$ surface, $G'\subset \Aut(Y)$ a finite group, $X$ the superspecial $K3$ surface in characteristic $p$ and $G\subset \Aut(X)$ a finite group. Assume that $G'_s\cong G_s$ is a Mukai group. We have seen that the action of $G$ on $\NS$ (respectively $G'$ on $\hyz$) lead to a $G$-invariant (resp. $G'$-invariant) orthogonal decomposition:

\begin{align*}
    \underbrace{\langle H \rangle\oplus T'}_{\subseteq \NS^{G_s}} &\oplus \NS_{G_s} \subset \NS , \\
    \underbrace{\langle H' \rangle\oplus T(Y)}_{\subseteq \hyz^{G'_s}} &\oplus \hyz_{G'_s} \subset \hyz  .
\end{align*}

\noindent Our goal is to show the two following claims:
\begin{Claimm}\label{Claim1}
    Fix a prime number $p$ and let $X$ be the superspecial $K3$ surface in characteristic $p$. Then for every extension $G\subset \Aut(X)$ of a Mukai group, there is a complex $K3$ surface $Y$, a subgroup $G'\subset \Aut(Y)$ with the same non-symplectic index as $G$ and $G_s\cong G'_s$ and in the decompositions above
    \[
    \langle H \rangle \cong \langle H'\rangle , \ \hyz_{G'_s}\cong \NS_{G_s}
    \]
    as well as 
    \[
    T(Y)\otimes \mathbb Z_q \cong T '\otimes \mathbb Z_q \ \text{ for every prime $q\neq p$}.
    \]
\end{Claimm}
\begin{Claimm} \label{Claim2}
    Let $Y$ be a complex $K3$ surface, $X$ the superspecial $K3$ surface in characteristic $p$ and $G'\subset \Aut(Y)$ the extension of a Mukai group $G_s\subset \Aut_s(X)$. Then the rank 1 lattice generated by the ample divisor $H'$ which is fixed by $G'$ embeds primitively into $\NS^{G_s}$ if and only if   
     \[
    \left ( \frac{\det(T(Y))}{p}\right) = -\left ( \frac{-1}{p}\right).
    \]
    In particular, the image of this embedding is a rank 1 lattice generated by an ample divisor $H$ in $\NS$ and there is an extension $G\subset\Aut(X)$ of $G_s$ such that
     \[
    \langle H \rangle \cong \langle H'\rangle , \ \hyz_{G'_s}\cong \NS_{G_s}
    \]
    as well as 
    \[
    T(Y)\otimes \mathbb Z_q \cong T '\otimes \mathbb Z_q \ \text{ for every prime $q\neq p$}.\]
\end{Claimm}
\noindent In the following, we will proof the existence of $Y$ in Claim \ref{Claim1}. Note that the existence of $G'\subset \Aut(Y)$ will be proved at a later point in Corollary \ref{Alldecompositions}.

\begin{Lemma}\label{embedding}
    Let $X$ be the superspecial $K3$ surface in characteristic $p$ and $G_s\subset \Aut_s(X)$ be a Mukai group such that $p \nmid \det(\NS_{G_s})$. Let $\langle H \rangle$ be a positive definite rank 1 lattice with $p\nmid H^2$. Assume $\langle H\rangle$ embeds primitively into $\NS^{G_s}$ with orthogonal complement $T'$. Then there exists $G_s\cong G'_s\subset O(E_8^{\oplus 2}\oplus U^{\oplus3})$ such that
    \begin{itemize}
        \item[a)] $\NS_{G_s}\cong (E_8^{\oplus 2}\oplus U^{\oplus3})_{G'_s}$,
        
        \item[b)] $\langle H\rangle$ embeds primitively into $(E_8^{\oplus 2}\oplus U^{\oplus3})^{G'_s}$,
        
        \item[c)] for the orthogonal complement $T\coloneqq\langle H\rangle ^\perp$ in $(E_8^{\oplus 2}\oplus U^{\oplus3})^{G'_s}$, we have \[T\otimes \mathbb Z_q \cong T '\otimes \mathbb Z_q \ \text{ for every prime $q\neq p$}.\]
    \end{itemize}

\end{Lemma}
\begin{Remark}
       Note that the genus of $(E_8^{\oplus 2}\oplus U^{\oplus3})^{G'_s}$ is uniquely determined by the choice of $G_s$, but there may appear two isometry classes in the genus. This corresponds to subgroups $G_1,G_2\subset O(E_8^{\oplus 2}\oplus U^{\oplus3})$ with $G_1\cong G_2\cong G_s$ such that $$(E_8^{\oplus 2}\oplus U^{\oplus3})^{G_1}\not \cong(E_8^{\oplus 2}\oplus U^{\oplus3})^{G_2}.$$ This can also be seen in results of Hashimoto \parencite{Hashimoto_2012}.\\
       The claim in Lemma \ref{embedding} is that $\langle H \rangle$ embeds into one of the isometry classes in the given genus, i.e. there is a subgroup $G'_s$ such that $\langle H \rangle$ embeds primitively in $(E_8^{\oplus 2}\oplus U^{\oplus3})^{G'_s}$. But this does not need to hold for every subgroup of $O(E_8^{\oplus 2}\oplus U^{\oplus3})$ which is isomorphic to $G_s$.
      
\end{Remark}
\begin{proof}[Proof of Lemma \ref{embedding}] In order to prove statement a), we will prove that such a $G'_s$ exists.\\
\textit{Existence of $G'_s$-action}: We recall that $\Lambda$ denotes the Leech lattice. By Theorem \ref{Leech}, there is a subgroup $\Tilde G_s\subset O(\Lambda)$ such that $\Lambda_{\Tilde G_s}(\cong \NS_{G_s})$. It is known that $\Lambda_{\Tilde G_s}$ embeds into the K3 lattice (see for instance \parencite[proof of Theorem 1.1]{brandhorst-hashimoto}). The isometry group $\Tilde G_s$ acts faithfully on $\Lambda_{\Tilde G_s}$ and induces the identity on the discriminant group, and thus can be extended with the identity on the orthogonal complement of $\Lambda_{\Tilde G_s}$ in $E_8^{\oplus 2}\oplus U^{\oplus3}$.
Thus we may extend the action of $\Tilde G_s$ on $\Lambda_{\Tilde G_s}$ to a subgroup $G'_s\subset O(E_8^{\oplus 2}\oplus U^{\oplus3})$ which is isomorphic to $G_s$ and whose coinvariant lattice $(E_8^{\oplus 2}\oplus U^{\oplus3})_{G'_s}$ is isomorphic to the coinvariant lattice $\NS_{G_s}$. This proves a).\\
The following proves both statement b) and c).\\
 \textit{$\langle H \rangle$ is primitive in an invariant lattice:} The statement is equivalent to proposing the existence of a positive definite rank 2 lattice $T$ such that $(E_8^{\oplus 2}\oplus U^{\oplus3})^{G'_s}$ is a primitive extension of $\langle H \rangle$ and $T$. This also automatically proves part b) of the Lemma. We will construct $T$ in a way which fulfills the properties in c).\\ 
 By assumption, there exists a negative definite lattice $T'$ such that $\NS^{G_s}$
is a primitive extension of $\langle H \rangle$ and $T'$. The local genus symbols of $\NS^{G_s}$ and $(E_8^{\oplus 2}\oplus U^{\oplus3})^{G'_s}$ coincide at any prime $q \neq p$. This implies that if the positive definite rank 2 lattice $T$ with the same local genus symbols as $T'$ at every prime $q \neq p$ exists, then $(E_8^{\oplus 2}\oplus U^{\oplus3})^{G'_s}$ is a primitive extension of $T$ and $\langle H \rangle$.\\
The existence of $T$ is proven using Theorem \ref{NikulinExistenceResult}. For this, we need to prove the existence of a positive definite rank 2 lattice, such that the quadratic form on the discriminant group coincides with the one of $T'$ for every prime $q\neq p$ and $p$ does not appear in the genus symbol of $T$.\\
    For this, note that the non-unimodular part of the genus symbol of $T'$ at $p$ coincides with the non-unimodular part of the genus symbol of $\NS$ at $p$, i.e.\ is of the form $p^{-\epsilon 2}$ with $\epsilon = \left( \frac{-1}{p} \right)$. There is a rank 24 lattice with signature $(21,3)$ obtained as the gluing of the lattices $\NS[-1]$ and $T'$ along their respective $p$ subgroups. The discriminant group of this lattice has the same quadratic form which was proposed for the discriminant group of $T$. By abuse of notation, we will denote the quadratic form of the primitive extension of $\NS[-1]$ and $T'$ as $q_T$. We can see that $\sign \ q_{T} = 21-3 \equiv 2 \mod 8$. \\
    By Theorem \ref{NikulinExistenceResult}, the requirements $(1)$ and $(2)$ for the existence of $T$ are thus fulfilled. For requirements $(3)$ and $(4)$ we note that the conditions only need to hold for primes $q\neq p$ since the new quadratic form $q_{T}$ does not contain a $p$-part. Furthermore, by our assumptions for $T$ we have that $\det(T') = p^2\det(T)$. Since $p$ is a unit in $\mathbb Z_q$, the congruences in $(3)$ and $(4)$ still hold because they hold for $q_{T'}$. This implies that all the conditions in Theorem \ref{NikulinExistenceResult} are fulfilled and thus the even positive definite rank 2 lattice $T$ exists.   
\end{proof}

\begin{Lemma}\label{fromptocomplex}
Let $X$ be the superspecial $K3$ surface in characteristic $p$ and $G_s\subset \Aut_s(X)$ be a Mukai group with $p \nmid \det(\NS_{G_s})$.
    Assume that there are lattices $T'$ and $\langle H \rangle$ such that
    \begin{itemize}
        \item $\langle H \rangle$ is a positive definite rank 1 lattice with $p\nmid H^2$
        \item  $T'$ is a negative definite rank 2 lattice with $p^2\mid \det(T')$
    \end{itemize}
    and $\NS^{G_s}$ is a primitive extension of $\langle H \rangle$ and $T'$. 
    Then there exists a complex $K3$ surface $Y$ such that $G_s\cong G'_s\subset \Aut(Y)$ and $H^2(Y,\mZ)^{G'_s}$ is a primitive extension of $\langle H \rangle$ and $T(Y)$.
    
\end{Lemma}
\begin{proof}
    The existence of a suitable lattice $T$ such that $(E_8^{\oplus 2}\oplus U^{\oplus3})^{G'_s}$ is a primitive extension of $T$ and $\langle H \rangle$ is proven in Lemma \ref{embedding}. It suffices to show that $T$ is the transcendental lattice of a complex $K3$ surface $Y$ and that $G'_s$ acts on $Y$. \\
    \indent \textit{Existence of $Y$:} By \parencite[Theorem 4]{Shioda_Inose_1977} $T$ is isomorphic to the transcendental lattice of a singular $K3$ surface. From now on, we will denote this transcendental lattice as $T(Y)$.

    \textit{Action on $Y$:} We know that $G'_s$ acts on the $K3$ lattice and fixes $T(Y)$. It remains to show that this action is induced by automorphisms, i.e.\ fulfills the conditions of the global Torelli theorem. For this, note that $\NSY$ is a primitive extension of $\langle H \rangle$ and $(E_8^{\oplus 2}\oplus U^{\oplus3})_{G'_s}$. The lattice $(E_8^{\oplus 2}\oplus U^{\oplus3})_{G'_s}$ embeds into the Leech lattice and thus contains no roots, i.e. elements with square $-2$. As a consequence, the divisor $H$ in $\NSY$ is contained in a chamber of the positive cone with respect to the decomposition by roots and we may assume that $H$ is ample. More details can for instance be found in \parencite[Chapter 8]{Huybrechts_2016}.
    Since $G_s$ induces the identity on $T(Y)$, fixes the ample divisor $H$ and preserves $\NS$, we can conclude from the global Torelli Theorem that $G_s$ is induced by automorphisms of $Y$.

\end{proof}
\begin{Remark}
    In the situation of Lemma \ref{fromptocomplex} we have
    \[
    [\NS^{G_s}:\langle H \rangle \oplus T' ] = [H^2(Y,\mZ)^{G'_s}: \langle H \rangle \oplus T(Y)].
    \]
\end{Remark}

\noindent We now aim to prove the decomposition of Claim \ref{Claim2}. Note that the existence of $G$ in this Claim will be proven in Corollary \ref{Alldecompositions}.

\begin{Lemma}\label{fromCtoP}
    Let Y be a complex $K3$ surface and $G'_s\subset \Aut_s(Y)$ be a Mukai group. Assume that $5\leq p\nmid \det(\hyz_{G_s})$ is a prime number such that 
    \[
    \left ( \frac{\det(T(Y))}{p}\right) = -\left ( \frac{-1}{p}\right).
    \]
  Let $\langle H \rangle$ denote the orthogonal complement of $T(Y)$ in $\hyz^{G'_s}$.
Let $X$ be the superspecial $K3$ surface in characteristic $p$. There is a subgroup $G'_s\cong G_s\subset\Aut(X)$ such that the following statements are true.
  \begin{itemize}
      \item[a)]  $\langle H \rangle$ embeds primitively into $\NS^{G_s}$
      \item[b)] Let $T'$ be the orthogonal complement of $\langle H\rangle$ in $\NS^{G_s}$. Then 
      \[T(Y)\otimes \mathbb Z_q \cong T '\otimes \mathbb Z_q \ \text{ for every prime $q\neq p$}.\]
  \end{itemize}
    \end{Lemma}

\begin{proof}
Our assumptions imply that $p\nmid \det(T(Y))$ and $p\nmid H^2$.\\
    \indent \textit{Existence of $G_s$}: It is proven in \parencite[Proposition 10.1]{ohashi2024finite} that a subgroup which is isomorphic to $G'_s$ is contained in $\Aut_s(X)$. Since we assume that the determinant of $ \NS_{G_s}\cong \hyz_{G'_s}$ is not divisible by $p$, we know that $\det(\NS^{G_s})=p^2 \cdot a$ for a integer $a$ coprime to $p$. The local genus symbol at $p$ of $\NS^{G_s}$ and $\NS$ coincide.
    
   \textit{Primitive embedding of} $\langle H \rangle$: We know that $\NS_{G_s}\cong \hyz_{G'_s}$ by Theorem \ref{Leech}. Thus we may assume that $\NS^{G_s}$ and $\hyz^{G'_s}$ have the same local genus symbol at every prime $q \neq p$. 
   Similarly to the proof of Lemma \ref{embedding}, it suffices to show the existence of a negative definite rank 2 lattice $T'$ such that $\NS^{G_s}$ is a primitive extension of $\langle H \rangle$ and $T'$. Furthermore, since $p\nmid H^2$, we may conclude that $T'$, if it exists, is a rescaling of a rank 2 lattice by $p$. In particular, the local genus symbol at $p$ of $\NS$ and $T'$ need to coincide.\\
   The existence of $T'$ is proven using Theorem \ref{NikulinExistenceResult}. Explicitly, the lattice $T(Y)\oplus \NS$ has a discriminant group that is isomorphic to the one we propose for $T'$. Thus the signature mod 8 of this quadratic form is $3-21\equiv 6\mod 8$. This shows that the proposed negative definite lattice $T'$ fulfills condition (1) of Theorem \ref{NikulinExistenceResult}.\\
   Furthermore, condition (2) is clearly fulfilled. Condition (3) and (4) are also fulfilled for every prime $q\neq p$ since it is fulfilled for $T(Y)$ (since we assume that $T(Y)$ exists) and the proposed determinant of $T(Y)$ and $T'$ differ by $p^2$, which is an element of $(\mathbb Z_q^*)^2$.\\
   It remains to show that condition (3) is also fulfilled at the prime $p$. 
   Due to the genus symbol of $\NS$ at $p$, the determinant of the $p$-adic lattice $\mathbb Z_p \otimes T'$ is a square in $\mathbb Z_p$ if and only $-1$ is a square in $\mathbb Z_p$. 
   By assumption, this is true since we propose that $\det(T(Y))=\det(T')/p^2$. This is equivalent to condition (3) and we may conclude that $T'$ exists.
\end{proof}

\begin{Remark}
    In the situation of Lemma \ref{fromCtoP}, taking the notation $T$ we have
    \[
    [\NS^{G_s}:\langle H \rangle \oplus T' ] = [H^2(Y,\mZ): \langle H \rangle \oplus T(Y)].
    \]
\end{Remark}

\begin{Lemma}\label{pIsLargeEnough}
Consider the setting of Lemma \ref{fromptocomplex}.
Assume that $$p\geq2\cdot [\langle H\rangle ^\perp: T'\oplus \NS_{G_s}].$$
Then a class spanning $\langle H \rangle$ in $\NS$ is ample.
\end{Lemma}
\begin{proof}
    In order to prove ampleness of a class spanning $\langle H \rangle$, it suffices to show that the orthogonal complement $\langle H \rangle ^\perp$ in $\NS$ does not contain an element with square $-2$ (i.e.\ it contains no roots).\\
    The lattice $\NS_{G_s}$ does not contain roots since it also embeds into the Leech lattice.
    The lattice $T'$ is a rescaling of a rank 2 negative definite lattice by $p$ and thus cannot contain any elements of square $-2$.\\
    We know that $\langle H \rangle ^\perp$ is a primitive extension of $T'$ and $\NS_{G_s}$. The elements of $\langle H \rangle ^\perp$ that are not contained in $T'\oplus \NS_{G_s}$ are of the form
    \[
    z = 
    \frac{x}{d}+\frac{y}{d}
    \]
    with $x\in T'$, $y\in \NS_{G_s}$ and an integer $d\mid [\langle H\rangle ^\perp: T'\oplus \NS_{G_s}]$ such that $x/d$ (resp. $y/d$) is contained in $\dul{(T')}$ (resp. $\dul{(\NS_{G_s})}$) and
    \[
    z^2= \frac{x^2}{d^2}+\frac{y^2}{d^2} \in 2\mathbb Z.
    \]
    We aim to get a more precise description of the value of $z^2$. For this, we note that since $x/d$ and $y/d$ are contained in the respective dual lattices, $x^2$ and $y^2$ must be divisible by $d$. Furthermore, since $T'$ is isomorphic to a lattice scaled by $p$, $x^2$ must be divisible by $p$.
    Hence we may write $x^2 = p\cdot d\cdot a$ and $y^2=d\cdot b$ for suitable integers $a,b<0$.
    We obtain
    \[
     z^2= \frac{x^2}{d^2}+\frac{y^2}{d^2} = \frac{p\cdot a+b}{d}\leq \frac{p\cdot a+b}{[\langle H\rangle ^\perp: T'\oplus \NS_{G_s}]}\stackrel{p\geq2\cdot [\langle H\rangle ^\perp: T'\oplus \NS_{G_s}]}{<}-2
    \]
    Hence, for large enough primes the primitive extension $\langle H \rangle ^\perp$ of $T'$ and $\NS_{G_s}$ does not contain any roots and thus a divisor representing a generator of $\langle H \rangle$ can be chosen to be ample.
\end{proof}

\begin{Remark}\label{PIsSmall}
    The assumption $$p\geq2\cdot [\langle H\rangle ^\perp: T'\oplus \NS_{G_s}]$$
    in Lemma \ref{pIsLargeEnough} can be removed by calculating the primitive extension $H^{\perp}$ of $T'$ and $\NS_{G_s}$ explicitly for every prime number below this bound. This leads to the conclusion that no such lattice contains a root, and thus we may assume a class that spans $\langle H \rangle$ to be ample.\\
    The associated computations were done in the computer algebra system Oscar \parencite{OSCAR}, \parencite{OSCAR-book} using code on 'Nikulin's theory on primitive embeddings' (developed by Stevell Muller).
\end{Remark}

\noindent In the following, we show that isometries of order 4 or 6 on $T'$ and $T(Y)$ extend to automorphisms of $X$ and $Y$ respectively when we restrict the index $[\NS^{G_s}:\langle H \rangle \oplus T' ] = [H^2(Y,\mZ): \langle H \rangle \oplus T(Y)]$ as in Lemma \ref{einschränkungIndex}.

\begin{Remark}
For every prime $q\neq p$, $l_q(A_{T(Y)})=2$ if and only if $l_q(A_{T'})=2$. Hence, for a suitable positive integer $a$ and positive definite rank 2 lattices $L$ and $L'$ we have $T(Y)\cong L[a]$ and $T'\cong L'[-pa]$. In particular $L$ and $L'$ have the same determinant.
\end{Remark}

\begin{Lemma}\label{indexmatches}
    Let $X$ be the superspecial $K3$ surface in characteristic $p$ and $Y$ a projective complex $K3$ surface such that $\Aut_s(Y)\supset G'_s\cong G_s \subset \Aut_s(X)$ is a Mukai group and $p\nmid \det\NS_{G_s}$. Assume that the rank 1 positive definite lattice $\langle H\rangle$ with $p\nmid H^2$ embeds primitively into $\NS^{G_s}$ and $\hyz^{G'_s}$ with orthogonal complement $T'$ and $T(Y)$ respectively, such that \[T(Y)\otimes \mathbb Z_q \cong T '\otimes \mathbb Z_q \ \text{ for every prime $q\neq p$}.\]
The following statements are true:
\begin{itemize}
    \item[1)] $O(T')$ contains an isometry of order 4 if and only if $O(T(Y))$ contains an isometry of order 4.
    \item[2)] $O(T')$ contains an isometry of order 6 if and only if $O(T(Y))$ contains an isometry of order 6.
\end{itemize}
\end{Lemma}
\begin{proof}
    1) This statement is equivalent to saying that $T'\cong L_{4}[-pa]$ if and only if $T(Y)\cong L_{4}[a]$ for a suitable positive integer $a$.\\
    Assume that $T'\cong L_{4}[-pa]$. Then $T(Y)=L[a]$ for a positive definite rank 2 lattice with determinant 4. But by Remark \ref{latticesnutzlich}, this implies that $L$ must be $L_{4}$. The converse follows analogously.\\
    2) This statement is equivalent to saying that $T'\cong L_6[-pa]$ if and only if $T(Y)\cong L_6[a]$ for a suitable integer $a$.\\
    Assume that $T'\cong L_6[-pa]$. Then $T(Y)=L[a]$ for a positive definite rank 2 lattice with determinant 3. Furthermore, since $L_6$ is even, $L$ must be even too (since the genus symbol at 2 is of type $\romantwo$). But by Remark \ref{latticesnutzlich}, this implies that $L$ must be $L_6$. The converse follows analogously.
\end{proof}
\begin{Remark}
    Generally, $T'$ is not a rescaling of $T(Y)$ by $-p$. This can be seen in the table in Remark \ref{l27}.
\end{Remark}

\begin{Proposition}\label{nothree}
    Let $X$ be the superspecial $K3$ surface in characteristic $p$, $G\subset\Aut(X)$ such that the symplectic part $G_s\neq G$ is a Mukai group and $p\nmid \det\NS_{G_s}$. Assume that $H $ is an ample divisor fixed by $G$ and $T'$ is the orthogonal complement of $\langle H \rangle$ in $\NS^{G_s}$. Then 
    \[
    [\NS^{G_s}:\langle H \rangle \oplus T' ] \in \{1,2\}.
    \]
\end{Proposition}
\begin{proof}
    By Lemma \ref{einschränkungIndex}, it suffices to show that $[\NS^{G_s}:\langle H \rangle \oplus T' ]\neq 3$. \\
    Assume the converse, i.e.\ $[\NS^{G_s}:\langle H \rangle \oplus T' ]= 3$. 
    Furthermore, there must be order 3 subgroups $H_1\subset A_{T'}$ and $H_2\subset A_{\langle H \rangle}$ with which we identify $\NS^{G_s}$ as a primitive extension of $\langle H \rangle$ and $T'$.\\
    Since we assume that $G\neq G_s$, the index $[G:G_s]=3$ (Lemma \ref{einschränkungIndex}).
    By Lemma \ref{twolattices}, we know that $T'\cong L_6[-pa]$ for a suitable $a$. 
    By our assumption, an order 3 isometry $h$ in $O(L_6[-pa])$ must extend to an  isometry in $O(\NS^{G_s})$ with the identity on $\langle H\rangle$. In particular, $h$ must map $H_1$ to itself and the restriction of $h$ to $H_1$ must coincide with the identity on $H_2$.\\
    By Lemma \ref{fromptocomplex}, for a suitable complex $K3$ surface $Y$ with $G_s\cong G'_s \subset \Aut_s(Y)$ we may assume that $\langle H\rangle$ is the orthogonal complement of $T(Y)$ in $\hyz^{G'_s}$. Furthermore, we know that $T(Y)\cong L_6[a]$ by Lemma \ref{indexmatches}. Since $T(Y)$ and $T'$ only differ by scaling, we know that $O(T')\cong O(T(Y))$.  There exists an embedding $\psi: H_1 \hookrightarrow A_{T(Y)}$ and $\hyz^{G'_s}$ is a primitive extension of $\langle H \rangle$ and $T(Y)$ associated to the subgroups $\psi(H_1)$ and $H_2$. Due to our assumptions, this implies that there is an order 3 isometry $h'$ in $O(T(Y))$ such that $h'$ induces the identity on $\psi(H_1)$, which can thus be extended to an isometry of $\hyz^{G_s}$ and consequently $g'$ of $\hyz$ by Corollary \ref{hashimotosurjective}.\\
    We claim that $g'$ is induced by an automorphism. It is clear that $g'$ preserves $\NSY$ and also fixes an ample divisor. Since an ample divisor is preserved, the isometry must thus preserve the Kähler cone. It remains to check that $g'$ preserves the Hodge structure, i.e.\ preserves $H^{2,0}$ and $H^{0,2}$. For this we note that $H^{2,0}$ and $H^{0,2}$ can be identified with the isotropic subspaces of $T(Y)$. One can check that the order 3 isometry $h'$ on $T(Y)$ preserve the isotropic subspaces of $T(Y)$ and thus $g'$ must preserve the Hodge structure and is induced by an automorphism.\\
    Since $T(Y)$ is the transcendental lattice associated to a complex $K3$ with a non-trivial extension of $G'_s$, $T(Y)$ must appear in the list of Brandhorst and Hashimoto \parencite[Section 6]{brandhorst-hashimoto} with a gluing along 3. But such a gluing does not occur, which leads to a contradiction.
\end{proof}

\noindent In the following Propositions, we will prove that some isometries of $T'$ will induce non-trivial extensions of $G_s$ in $\Aut(X)$.
\begin{Proposition}\label{index4}
    Let $Y$ be a complex $K3$ surface and $G'_s\subset \Aut_s(Y)$ a Mukai group. Assume that there is a finite group $G'\subset \Aut(Y)$ with symplectic subgroup $G'_s$ and non-symplectic index 4. \\
    Let $X$ be the superspecial $K3$ surface in characteristic $p\equiv 3 \mod 4$. Then there is a subgroup $G\subset \Aut(X)$ with non-symplectic index 4 and symplectic subgroup $G_s\cong G'_s$. 
 
\end{Proposition}

\begin{proof}
    Using Proposition \ref{simonhashimotoresult}, we may conclude that $T(Y)$ is a scaling of $L_{4}$. By Lemma \ref{fromCtoP}, the sublattice $\langle H \rangle$ spanned by the primitive ample divisor $H$ which is fixed by $G'$ on $Y$ embeds into $\NS^{G_s} $ for the superspecial $K3$ surface in characteristic $p\equiv 3\mod 4$. 
    By Lemma \ref{pIsLargeEnough} and Remark \ref{PIsSmall} we may assume that the embedding of $\langle H \rangle$ in $\NS$ contains an ample class.\\
    By the proof of Lemma \ref{indexmatches}, we see that the orthogonal complement $T'$ of $\langle H \rangle$ in $\NS^{G_s}$ is isomorphic to $T(Y)[-p]$. Denote the order 4 automorphism of $T'$ as $_{L_{4}[-pa]}$. 
    \begin{Claim}
        The isometry $g_{L_{4}[-pa]}$ can be extended to an isometry $g_{\NS}$ of $\NS$ which fixes the ample classes in $\langle H \rangle$.
    \end{Claim}
    \begin{proof}[Proof of Claim]
    We obtain from Proposition \ref{nothree} that \[
    [\NS^{G_s}:\langle H \rangle \oplus T']=[\hyz^{G'_s}:\langle H \rangle \oplus T(Y)]\in \{1,2 \}.
    \]   
    By the assumption that $G'_s$ has a non-trivial extension with non-symplectic index 4 on $Y$, we know that the order 4 isometry on $T(Y)$ can be extended with the identity on $\langle H \rangle$ to an isometry $g_{L_{D_4[a]}}$ of $\hyz^{G'_s}$ (and thus of $\hyz$ by Corollary \ref{hashimotosurjective}). In the case where $[\hyz^{G'_s}:\langle H \rangle \oplus T(Y)]=2$, extending the isometry to $\hyz^{G'_s}$ implies that $g_{L_{D_4[a]}}$ induces the identity on the subgroup which we glue along.
    Let $x,y$ be two generators of $T(Y)$ such that the induced Gram matrix is of the form
    \[
    \begin{pmatrix}
        2a & 0 \\0& 2a
    \end{pmatrix}.    \]    
    One can calculate that the only order 2 subgroup $\Gamma$ which is fixed by $g_{L_{D_4[a]}}$ can be represented by the two equivalence classes $\{0,(1/2)x+(1/2)y \}$. When rescaling $T(Y)$ with $-p$, we do not change the quadratic form of this subgroup since $p\equiv 3 \mod 4$. Hence, we may also extend the isometry $g_{L_{4}[-pa]}$ to an isometry of $\NS^{G_s}$ by gluing with the identity on $\langle H \rangle$ along $\Gamma$. In the case where no gluing occurs, we may always extend with the identity on $\langle H \rangle$ and obtain an order 4 isometry of $\NS^{G_s}$ which we may denote as $g_{\NS^{G_s}}$. \\
    When gluing $\NS^{G_s}$ and $\NS_{G_s}$ into $\NS$, we glue along the largest subgroup of $\discgr{\NS^{G_s}}$ of order coprime to $p$ and the full discriminant group of $\NS_{G_s}$. Due to Theorem \ref{Hashimotosurj}, there is an isometry of $\NS_{G_s}$ with which we can extend $g_{\NS^{G_s}}$ to $\NS$.
    \end{proof}   
    \begin{Claim}
        The isometry $g_\NS$ is induced by an automorphism of $X$.
    \end{Claim}
    \begin{proof}[Proof of Claim]
         By our construction, $g_{\NS}$ preserves the ample class spanning $\langle H \rangle$. Furthermore, since $G_\NS$ preserves an ample class, the isometry must also preserve the ample cone since chambers of the Weyl decomposition of the positive cone must be mapped to chambers. Since the ample cone is a chamber of this decomposition, the isometry must preserve the ample cone.
         So, it remains to check that $(g_\NS)_A\otimes k$ preserves the strictly characteristic subspaces, as is needed to apply the crystalline Torelli theorem.\\
    For this, we note that the induced action of $g_\NS$ on the discriminant group coincides with the action of  $g_{L_{4}[-pa]}$ on the $p$-part of the discriminant group of $A_{L_{4}[-pa]}$ since no other sublattice in the decomposition has a discriminant group with a non-trivial $p$-part and thus also no gluing appears. \\
    Consequently, the property of $(g_\NS)_A\otimes k$ preserving the characteristic subspaces only depends on the restriction of $g_A$ to the $p$-Sylow group of $A_{L_4[-pa]}$. This action is independent of the choice of $a$ since $l_p(L_4[-pa])= \text{rk}\ L_4[-pa]$. Furthermore, $\langle H \rangle $ and $G_s$ also do not influence this action. \\
    In order to check that $(g_\NS)_A\otimes k$ preserves the characteristic subspace, it thus suffices to show that for every $p\equiv 3 \mod 4$ that there is a K3 surface $Y$ with superspecial reduction mod $p\equiv 3 \mod 4$, a Mukai group $G'_s\Aut_s(Y)$ such that there is an extension $G'\subset \Aut(Y)$ with non-symplectic index 4 which also acts on the superspecial reduction.\\
    Model \parencite[80 a)]{brandhorst-hashimoto}: An action of $F_{384}$ with an index 4 extension is given on the Fermat quartic 
\[
V(x^4+y^4+z^4+w^4) \subset \mathbb P^3_{\mathbb C}.
\]
This surface has good reduction for $p \neq 2$. The action of $F_{384}$ is shown in \parencite[no. 5]{Mukai1988}. The group action of the index 4 extension of $F_{384}$ is generated by the permutation of coordinates and projective automorphisms of the form 
\[
(x:y:z:w) \mapsto(i^ax:i^by:i^cz:i^dw)
\]
where $i$ is a fourth root of unity and $1\leq a,b,c,d \leq4 $. These automorphisms also generate a group with the same structure in any characteristic $p\neq 2$.\\
The surface is isomorphic to the superspecial $K3$ surface when reducing at a prime $p\equiv 3 \mod 4$, as can for instance be shown using \parencite[Proposition 4.1]{Schütt_2016}.
    \end{proof}
\noindent This proves that $g_\NS$ is induced by an automorphism with non-symplectic index 4. Additionally, $g_\NS$ fixes an ample divisor which also fixed by $G_s$, and thus
\end{proof}

\begin{Proposition}\label{index6}

    Let $Y$ be a complex $K3$ surface and $G'_s\subset \Aut_s(Y)$ a Mukai group. Assume that there is a finite group $G'\subset \Aut(Y)$ with symplectic subgroup $G'_s$ and non-symplectic index 6. \\
    Let $X$ be the superspecial $K3$ surface in characteristic $p\equiv 5 \mod 6$. Then there is a subgroup $G\subset \Aut(X)$ with non-symplectic index 6 and symplectic subgroup $G_s\cong G'_s$.

\end{Proposition}
\begin{proof}
  
Due to an analogous argument to Proposition \ref{index4}, it suffices to prove that there exists a complex $K3$ surface with an action of a Mukai group, such that this group has an index 6 extension, the $K3$ surface has supersingular reduction and inherits the action of the index 6 extension. For this, we consider the following model from \parencite{brandhorst-hashimoto}:\\
    Model \parencite[63 a)]{brandhorst-hashimoto}: The group $M_9$ acts with an index 6 extension on the complex $K3$ surface defined by
\[
V(x^6+y^6+z^6-10(x^3y^3+y^3z^3+z^3x^3) = (w')^2) \subset \mathbb P(1,1,1,3)_\mathbb C. 
\]
This surface has good reduction for $p \neq 2,3$.\\
The action of $M_9$ is described in \parencite[no. 10]{Mukai1988}. The group action of the index 6 extension of $M_9$ is generated by the matrices
\begin{align*}
   & \begin{pmatrix}
    1 & &&\\ & w&& \\ & & w^2&\\&&&1
\end{pmatrix},
\begin{pmatrix}
    & 1& &\\ &&1&\\1&&&\\&&&1
\end{pmatrix},
\frac{1}{w-w^2}\begin{pmatrix}
    1 & 1&1&\\1 &w&w^2&\\1&w^2&w&\\&&&w-w^2
\end{pmatrix},\\
&\frac{1}{w-w^2}\begin{pmatrix}
    1 & w&w&\\w^2 &w&w^2&\\w^2&w^2&w&\\&&&w-w^2
\end{pmatrix},\begin{pmatrix}
    w & & & \\ & 1& & \\ & & 1&\\
    &&&-1
\end{pmatrix}.
\end{align*}
Here, $w$ is a third root of unity. These automorphisms also generate a group of the same structure over fields of characteristic not equal to two or three.\\
The surface has supersingular reduction for $p \equiv 5 \mod 6$, as can for instance be seen by \parencite[Proposition 4.1]{Schütt_2016}.
\end{proof}
\begin{Proposition}\label{index2}
    Let $X$ be the superspecial $K3$ surface in characteristic $p$. Assume that $G_s\subset\Aut_s(X)$ is a Mukai group such that $p\nmid \NS_{G_s}$ and $\NS^{G_s}$ is a primitive extension of $\langle H \rangle \oplus T'$ for a suitable negative definite lattice $T'$
    with $l_p(T')=2$ and $[\NS^{G_s}:(\langle H\rangle \oplus  T')] \in \{1,2 \}$. Then there exists a group $G\subset \Aut(X)$ with symplectic subgroup $G_s$ such that $[G:G_s]=2$.
\end{Proposition}
\begin{proof}
The isometry $-\id \in O(T')$ can be extended to an involution on $\NS^{G_s}$ by taking the identity on $\langle H \rangle$. When considering \NS \ as the primitive extension of $\NS^{G_s}$ and $\NS_{G_s}$, we glue along the full discriminant group of $\NS_{G_s}$. Thus, the isometry on $\NS_{G_s}$ can be extended to an isometry $g$ on $\NS$ by Theorem \ref{hashimotosurjective}. Due to our assumptions, $p\nmid H^2$ and thus every element of $T'$ is not dually primitive in $\NS$. Since we induce $-\id$ on $A_{T'}$, we may conclude that this action induces $-\id$ on $A_\NS$. In particular, this action preserves the characteristic subspace and $g$ is induced by an automorphism. 
\end{proof}

\begin{Corollary}\label{Alldecompositions}
    Let $X$ be the superspecial $K3$ surface in characteristic $p$ and $Y$ a complex $K3$ surface such that $\Aut_s(Y)\supset G'_s\cong G_s \subset \Aut_s(X)$ is a Mukai group and $p\nmid \det\NS_{G_s}$. Assume that there is a rank 1 positive definite lattice $\langle H\rangle$ such that 
    \begin{itemize}
        \item[-]  $p\nmid H^2$
        \item[-] $\langle H \rangle $ embeds primitively into $\NS^{G_s}$ and $\hyz^{G'_s}$ with orthogonal complement $T'$ and $T(Y)$ respectively
        \item[-] $T(Y)\otimes \mathbb Z_q \cong T '\otimes \mathbb Z_q \ \text{ for every prime $q\neq p$}$
    \end{itemize}
    \noindent Then the following statements are true:
    \begin{itemize}
        \item[1)] Assume that $T(Y)$ (and thus $T'$, \ref{indexmatches}) is a scaling of $L_6$ and 
        \[
    [\NS^{G_s}:\langle H \rangle \oplus T' ] = [H^2(Y,\mZ): \langle H \rangle \oplus T(Y)]=1.
    \] 
    Then there are extensions $G \subset \Aut(X),G'\subset \Aut(Y)$ with non-symplectic index 6.
    \item[2)] Assume that $T(Y)$ (and thus $T'$, \ref{indexmatches}) is a scaling of $L_{4}$ and 
        \[
    [\NS^{G_s}:\langle H \rangle \oplus T' ] = [H^2(Y,\mZ): \langle H \rangle \oplus T(Y)]\in \{ 1,2\}.
    \] 
    Then there are extensions $G \subset \Aut(X),G'\subset \Aut(Y)$ with non-symplectic index 4.
    \item[3)] Assume that
        \[
    [\NS^{G_s}:\langle H \rangle \oplus T' ] = [H^2(Y,\mZ): \langle H \rangle \oplus T(Y)]\in \{ 1,2\}.
    \] 
    Then there are extensions $G \subset \Aut(X),G'\subset \Aut(Y)$ with non-symplectic index 2.
    \end{itemize}
\end{Corollary}
\begin{proof}
    This follows directly from the results of Brandhorst and Hashimoto \parencite[Section 6]{brandhorst-hashimoto} and Propositions \ref{index4}, \ref{index6} and \ref{index2}.
\end{proof}

\begin{proof}[Proof of Theorem \ref{CorrespondenceComplexSupSpe}]
    Assume that $G\subset \Aut(X)$ is a group such that $G_s\neq G$ is the symplectic subgroup. Then $G$ fixes an ample divisor $H$. Let $T'$ denote the orthogonal complement of $\langle H \rangle$ in $\NS^{G_s}$. By Lemma \ref{einschränkungIndex} and Proposition \ref{nothree}, we know that $[\NS^{G_s}:\langle H \rangle \oplus T']\in \{1,2 \}$. The proposed complex $K3$ surface $Y$ exists due to Lemma \ref{fromptocomplex} and there is a subgroup $G'$ of the same non-symplectic index as $G$ due to Corollary \ref{Alldecompositions}.\\
    For the converse, assume that the complex $K3$ surface $Y$ exists and $G'\subset \Aut(Y)$ contains a Mukai group $G'_s$. Due to Lemma \ref{fromCtoP}, the ample divisor fixed by $G'$ primitively embeds into $\NS^{G_s}$ with orthogonal complement $T'$. Due to Proposition \ref{simonhashimotoresult}, $[\NS^{G_s}:\langle H \rangle \oplus T']\in \{1,2 \}$ and thus by Corollary \ref{Alldecompositions}, there is a subgroup $G\subset\Aut(X)$ with symplectic subgroup $G_s$ and the same non-symplectic index as $G'$.
\end{proof}

\begin{proof}[Proof of Theorem \ref{Mukaimainresult}] By Theorem \ref{CorrespondenceComplexSupSpe}, it suffices to calculate for each complex $K3$ surface $Y$ in \parencite[Section 6]{brandhorst-hashimoto}, for which primes the condition 
    \[
    \left ( \frac{\det(T(Y))}{p}\right) = - \left ( \frac{-1}{p}\right)
    \]
is fulfilled. For such primes, we then know that the superspecial $K3$ surface in this characteristic has an action of an extension of the same non-symplectic index. As an example, we calculate this for the group $T_{48}$:\\
There are three complex $K3$ surfaces with an action of a non-trivial extension of $T_{48}$. Their transcendental lattices are listed below, as well as the largest possible index $n$ of an extension of $G_s$, the value of $H^2$ and the primes for which we have $\left ( \frac{\det(T(Y))}{p}\right) = -\left ( \frac{-1}{p}\right)$.
    \begin{center}
        \begin{tabular}{|c|c|c|c|}
        \hline
             $H^2$&$T(Y)$&$n$& conditions on $p$  \\
             \hline 
             2 & $\begin{pmatrix}
                 16 & 8 \\ 8 & 16
             \end{pmatrix}$ & 6 & $p\equiv 5 \mod 6$ \\
             16 & $\begin{pmatrix}
                 2 & 0 \\ 0 &48
             \end{pmatrix}$ & 2& $p \equiv 13,17,19,23 \mod 24$\\
             48 & $\begin{pmatrix}
                 2 & 0 \\ 0&16
             \end{pmatrix}$ & 2& $p \equiv 5,7 \mod 8$\\
             \hline
        \end{tabular}
    \end{center}
    As a consequence, for every group $G\subset \Aut(X)$ that contains $T_{48}$ as its symplectic subgroup the following holds for the non-symplectic index $n$:
    \begin{itemize}
        \item $n\in \{1,2,3,6\}$ if $p\equiv 5\mod 6$.
        \item $n\in\{1,2 \}$ if  $p\equiv 7,13,19,23 \mod 24$.
        \item $n=1$ if $p\equiv 1 \mod 24$.
    \end{itemize}
    Furthermore, there is a subgroup $G$ such that every possible value of $n$ in this list is obtained.
    We may repeat this process for every Mukai group and obtain the table in Theorem \ref{Mukaimainresult}.
\end{proof}

\begin{Remark}\label{l27}
     Consider the group $L_2(7)$. As is shown in \parencite{brandhorst-hashimoto}, there are two singular $K3$ surfaces with an action by an index 2 extension of $L_2(7)$ which have transcendental lattices of the same genus which are not isomorphic. As a consequence, these $K3$ surfaces are not isomorphic as well (see \parencite[Theorem 4]{Shioda_Inose_1977}]). \\
    But for a prime $p\equiv 11,29,37,43,51,53 \mod 56$, the associated lattices $T'$ are isomorphic.

\footnotesize
\begin{center}
    \begin{tabular}{|c|c|c|c|c|c|}
    \hline
    $H^2$&$T(Y)$&genus of $T(Y)$  &possible $T'$& genus of $T'$ & conditions on $p$\\
    \hline

        14&
        $\begin{pmatrix}
        2 & 0 \\ 0 & 28
        \end{pmatrix}$ & $[2^{+1} 4^{+1}]_0 7^{+1}$
         
        &    $\begin{pmatrix}
        -6 & -2 \\ -2 & -10
        \end{pmatrix}$
        & $[2^{+1} 4^{+1}]_4 7^{+1}$

        & $p \equiv 11,29,37,43,51,53 \mod 56$\\
        &&& $\left\{\begin{matrix}
        
    \begin{pmatrix}
             -2 &  0 \\
  0& -28 
        \end{pmatrix}\\ \begin{pmatrix}
            -4  & 0\\
  0 &-14
        \end{pmatrix}\end{matrix}\right \}
        $ & $[2^{+1} 4^{+1}]_0 7^{-1}$
        
        & $p \equiv 17,31,33,41,47,55 \mod 56$ 
        \\

        \hline

        14 & $\begin{pmatrix}
            4 & 0 \\0 &14
        \end{pmatrix}$
        &$[2^{+1} 4^{+1}]_0 7^{+1}$  & $\begin{pmatrix}
        -6 & -2 \\ -2 & -10
        \end{pmatrix}$
        & $[2^{+1} 4^{+1}]_4 7^{+1}$
     & $p \equiv 11,29,37,43,51,53 \mod 56$
    \\
     &&& $\left\{\begin{matrix}
        
    \begin{pmatrix}
             -2 &  0 \\
  0& -28 
        \end{pmatrix}\\ \begin{pmatrix}
            -4  & 0\\
  0 &-14
        \end{pmatrix}\end{matrix}\right \}
        $ & $[2^{+1} 4^{+1}]_0 7^{-1}$
        
        & $p \equiv 17,31,33,41,47,55 \mod 56$ 
        \\

    \hline
    
    \end{tabular}

\end{center} 
\normalsize
    This is the only example of a Mukai group $G_s$ where there are two non-isomorphic lattices $T(Y)$ which induce the same lattice $T'$.
\end{Remark}

\subsection{Proof of Theorems \ref{mainmainresult}, \ref{TameCase} and \ref{alotofcongruences} }\label{mainmainresultproof}
We will collect the proofs of the main Results \ref{mainmainresult}, \ref{TameCase} and \ref{alotofcongruences}. We will also briefly summarize the methods used.

\begin{proof}[Proof of Theorem \ref{TameCase}]
As has been proven by Ohashi and Schütt, the maximal symplectic groups acting in characteristic $p >11$ are either maximal groups that can be found in the classification of Mukai or finite, saturated subgroups $G_s \subset \Aut_s(X)$ associated to a invariant lattice $\NS^{G_s}$ of rank $2$ (Lemma \ref{asfixedpointlattice}). \\
To prove (i), we have considered all groups which act tamely on $X$ and have a coinvariant lattice of rank 20 in the classification of Höhn and Mason \parencite{HOHN2016618}. The possible extensions were then investigated in Section \ref{fourorbits}. Due to the location of the discriminant group in the decomposition into the invariant lattice and co-invariant lattice, we concluded that no extension can occur (Proposition \ref{rk2222}).\\
To prove (ii), we proved a stronger statement in Theorem \ref{Mukaimainresult}. For this, we related extensions of Mukai groups (i.e.\ the maximal groups over $\mathbb C$) on certain complex $K3$ surfaces with extensions of Mukai groups on the superspecial $K3$ surfaces (Theorem \ref{CorrespondenceComplexSupSpe}). It is then a direct computation to obtain the table in case (ii). 
\end{proof}
\begin{proof}[Proof of Theorem \ref{mainmainresult}]
    
    Let $G_s$ be a finite symplectic group acting on a $K3$ surface in characteristic $p> 11$ (i.e.\ $G_s$ is tame). Assume that no faithful action of $G_s$ appears on a complex $K3$ surface. Due to Theorem 1.1 in \parencite{ohashi2024finite}, we know that such groups $G_s$ only act on the superspecial $K3$ surface in this characteristic and on no other $K3$ surface. Consequently, it suffices to show the statement of Theorem \ref{mainmainresult} for the superspecial $K3$ surface $X$, i.e.\ show that the action of $G_s$ on $X$ only has trivial extensions.\\
    Additionally, in Proposition \ref{rk2222} we have proved that maximal groups that do not appear over the complex numbers have no non-trivial extension. It remains to show that this also holds true if the symplectic group is not maximal but does not appear over the complex numbers. We will denote this non-maximal group as $\Gamma_s$ and assume that $G_s$ is a maximal group which contains $\Gamma_s$ (or, equivalently, that $G_s$ is the saturation of $\Gamma_s$, see Section \ref{MatthiasandOhashi}).\\
    The groups that do not occur over the complex numbers are groups $\Gamma_s$ such that $\NS_{\Gamma_s}=\NS_{G_s}$ for a maximal group $G_s$ that does not occur over the complex numbers. Otherwise, the invariant lattice of $\NS^{\Gamma_s}$ would have rank $\geq 3$ and thus would be contained in a maximal group over $\mathbb C$ (Lemma \ref{asfixedpointlattice}).\\
    But then, the same argument as in the proof of Proposition \ref{rk2222} may be applied to $\Gamma_s$ and we may conclude that there are no non-trivial extensions of $\Gamma_s$.    
\end{proof}

\begin{proof}[Proof of Theorem \ref{alotofcongruences}] Let $G_s$ be a maximal symplectic group in $\text{Aut}(X)$. If $G_s$ is a group that does not appear over the complex numbers, we know it has no non-trivial extensions. As a consequence, in order to always have a non-trivial extension it is a necessary condition that the Mukai classification holds.
By \parencite[Theorem 1.6]{ohashi2024finite}, this is the case if and only if $p \equiv \pm(1^2),\pm(11^2),\pm(13^2),\pm(17^2),\pm(19^2),\pm(23^2)\mod 840$.\\Furthermore, we do not only need that the Mukai classification holds, but also that all Mukai groups have a non-trivial extension. This leads to the condition \\$p\equiv -1,-(11^2),-(13^2),-(17^2),-(19^2),-(23^2)\mod 840$.  
\end{proof}

\section{Extensions of wild actions}\label{wild}
Let $p \leq 11$. We exhibit the maximal extensions of the wild maximal symplectic groups in \parencite{ohashi2024finite}. The tame maximal group acting in characteristic 11 only has trivial extensions by Proposition \ref{rk2222}. All other groups with maximal symplectic actions in characteristic $\leq 11$ are wild.\\
We show which groups only have trivial extensions and give explicit realizations for the ones that admit non-trivial extensions. We also argue why these extensions have the optimal non-symplectic index for the given group. The goal is to prove Result \ref{wildcase}.

In the classification of Ohashi and Schütt \parencite[Sections 7 and 9]{ohashi2024finite}, the wild groups acting on the superspecial $K3$ surfaces in characteristic $\leq 11$ have an associated invariant lattice of rank 1 or 2.
 
\subsection{Trivial extensions}\label{rk1noext}
\subsubsection{\texorpdfstring{Characteristic $p\geq 5$}{Characteristic p greater than 3}}
We claim that there are groups with an associated invariant lattice of rank 1 which do not possess a non-trivial extension. Namely, we will consider the groups in the table below, which also contains the information of the characteristic $p$ they appear in and the local genus symbols of the invariant and coinvariant lattice.  
\begin{center}{\rowcolors{2}{white}{lightgray}
\begin{tabular}{ |c|c|c|c| }
\hline
$G_s$& $p$  & local genus symbol of $\NS^{G_s}$ & local genus symbol of $\NS_{G_s}$\\
\hline
  $\mathfrak A_8$ & $5$ &  $4^{+1}_73^{+1}5^{-1}$ & $4^{+1}_1 3^{+1} 5^{+1}$\\
 $2^4:(3 \times \mathfrak A_5):2$ & $5$ & $8^{+1}_7 3^{+1}5^{+1}$ & $8^{+1}_1 3^{-1}5^{-1}$\\
 $L_3(4).2$ & $7$ & $4^{-1}_53^{+1}7^{-1}$&$4^{-1}_33^{-1}7^{-1}$\\
 $2^4.\mathfrak A_7$ & $7$ &  $8^{+1}_77^{+1}$ &$8^{+1}_17^{+1}$\\
 $M_{22}$ & $11$ &  $4^{-1}_3 11^{+1}$ & $4^{-1}_5 11^{+1}$\\
 $M_{11}$ & $11$ &  $2^{+1}_1 3^{+1}11^{-1}$&$2^{+1}_7 3^{-1}11^{-1}$\\
\hline
\end{tabular}}
\captionof{table}{Maximal wild groups without extensions in characteristic $\geq5$}\label{maximalrk1} 
\end{center}

\noindent These groups are exactly the groups in \parencite[Section 7]{ohashi2024finite} where $p\geq 5$ divides the order of the discriminant group. Each group has a invariant lattice of rank 1. Some of these actions are also noted by Kond\={o} \parencite{kondo2005}.\\
Let $G_s$ be a group from Table \ref{maximalrk1} and $X$ be the superspecial $K3$ surface in the associated characteristic $p$. Assume that $G\subset \Aut(X)$ is a finite group with symplectic part $G_s$.\\  
By our assumption, $G$ fixes an ample divisor $H$, so the rank of $\NS^G$ is at least 1. As a consequence, $G$ must fix the generator of $\NS^{G_s}$ and thus we may conclude $\NS^{G_s}=\NS^{G}$. Furthermore, both the coinvariant and invariant lattices of the groups in Table \ref{maximalrk1} have a discriminant groups whose $p$-Sylow groups are isomorphic to $\mathbb Z/p\mathbb Z$. From this, we can see that no gluing along $p$ occurs.
This implies that the generator of $\NS^{G_s}$ is not dually primitive in $\NS$ and is fixed by any extension of $G_s$.\\
Since all groups in table 6 act in characteristic $p\geq 5$, we may apply Lemma \ref{AusschlussDuallyPrimitive}. Hence, $G = G_s$ must hold and there exists no non-trivial extension of these groups. We may conclude this case by stating the following Theorem.
\begin{Theorem}
    Let $G$ be a finite group acting faithfully on the superspecial $K3$ surface in characteristic $p$, such that $G_s$ is a maximal group in Table \ref{maximalrk1}. Then $G = G_s$.
\end{Theorem}

\subsubsection{Characteristic 2}
In characteristic 2, there is a wild automorphism group of the superspecial $K3$ surface that has no non-trivial extensions.
\begin{Example}
    Consider $G_s=\text{Aut}(\mathfrak S_6)$ with $\vert G_s \vert = 1440$. This group acts as a maximal group of symplectic automorphisms on the superspecial $K3$ surface in characteristic 2. 
    The co-invariant lattice $\NS_{G_s}$ is a rank 21 lattice with local genus symbol $2^{+3}_5 3^{-1} 5^{+1}$ as can be seen in \parencite[no. 194]{HOHN2016618}. Since the invariant lattice has rank 1, this implies that we glue along a subgroup of order 30. In particular, there are two independent elements $x,y$ of $\NS_{G_s}$ that are not dually primitive in $\NS$ and the discriminant group $A_\NS$ is generated by rational multiples of $x,y$. In other words, $A_\NS$ can be considered as a subgroup of $A_{\NS_{G_s}}$. 
    \\
    Assume that $G_s\subset G\subset\Aut(X)$ is a non-symplectic finite group. It is known that $G$ must act on the coinvariant lattice of $G_s$, i.e.\ there is a map $\psi\colon G\rightarrow O(\NS_{G_s})$. Since $G$ induces the identity on $\NS^{G_s}$, this map is an embedding. The non-symplectic index $[G:G_s]$ then has to divide $[O(\NS_{G_s}):\psi(G_s)]$.

    Using the coinvariant lattice from the supplementary files of \parencite{HOHN2016618}, one can may compute the isometry group of $\NS_{G_s}$ (this is a group isomorphic to the isometry group of $\Lambda_{G_s}$ due to Corollary \ref{Leech}). The order of $O(\Lambda_{G_s})$ is 11520, so the non-symplectic index of any extension divides $11520/1440=8$. With Theorem \ref{NonSympIndex} we may conclude that the non-symplectic index must be one.
\end{Example}

\subsection{Explicit Realizations}\label{Char23}

\begin{Possibility}\label{AusführlichesBeispiel}
    
Let $G_s = U_3(5)$. This group acts symplectically on the superspecial $K3$ surface in characteristic 5, as has been shown by Dolgachev and Keum \parencite{DolgachevKeum2009}. This action can be seen on the model
\[
X = V(x^6+y^6+z^6-w^2)\subset \mathbb P(1,1,1,3).
\]
Then $\textsc{GU}(3,\mathbb F_{5^2})$ is the group of matrices over $\mathbb F_{5^2}$ that leave the hermitian form $x^6+y^6+z^6$ invariant. We may extend the action of $\textsc{GU}(3,\mathbb F_{5^2})$ to $X$ by letting each matrix act on the coordinates $(x,y,z)$ and as the identity on $w$.  \\
Consider the subgroup $\textsc{SU}(3,\mathbb F_{5^2})$, where the determinant of the matrix is 1. This group contains multiples of the identity matrix. They are $\zeta_3\cdot \mathbb E_3$ and $(\zeta_3^2)\cdot \mathbb E_3$, with $\zeta_3$ being a primitive third root of unity and $\mathbb E_3$ the identity matrix. These matrices induce the identity when applying them in weighted projective space since the following holds
\begin{align*}
     ( \zeta_3^i)\cdot\mathbb E_3([x:y:z:w]) = [\zeta_3^i x: \zeta_3^i y:\zeta_3^i z: \underbrace{w}_{= \zeta_3^{3i}w}] = [x : y:z:w]. 
\end{align*}
This implies that the action of the quotient by these matrices is well-defined.
We obtain a symplectic action of $U_3(5)= \text{PSU}(3,\mathbb F_{5^2})$ on $X$. 
Additionally, the action must be symplectic, since the group is simple.\\

We claim that the quotient $G\coloneqq\textsc{GU}(3,\mathbb F_{5^2})/\langle (\zeta_3)\cdot\mathbb E_3\rangle $ is an index 6 extension of $G_s$. 
We see that $G$ induces a faithful action on $X$, since $\textsc{GU}(3,\mathbb F_{5^2})$ preserves the defining equation of $X$ and elements which induce the identity are removed by taking the quotient. Furthermore, $G$ clearly contains $G_s$ and thus the symplectic part of $G$ must be equal to $G_s$ since $G_s$ is a maximal symplectic group in characteristic 5.
We may conclude that $G$ is an extension of $G_s$ with non-symplectic index 6. This is the highest possible non-symplectic index in characteristic 5 due to Theorem \ref{NonSympIndex}. An isomorphic description of $G$ is $\text{PGU}(3,\mathbb F_{5^2}): 2 $.

\noindent This finite group action is also described by Katsura, Kond\=o and Shimada in \parencite{Katsura_Kondo_Shimada_2014}.

\end{Possibility}

\begin{Possibility} The group $G_s = U_4(3)$ acts symplectically on the superspecial $K3$ surface in characteristic 3. This is a classical example that can be found in \parencite{DolgachevKeum2009}, but has been known beforehand to Beauville, Mukai, Shioda and Tate. The action can be realized on the Fermat quartic
\[
X = V(x^4+y^4+z^4+w^4) \subset \mathbb P^3.
\]
The group $\text{GU}(4,\mathbb F_{3^2})$ consists of matrices over $\mathbb F_{3^2}$ that leave the Hermitian form $x^4+y^4+z^4+w^4$ invariant. The group $U_4(3)= \text{PSU}(4,\mathbb F_{3^2})$ (i.e.\ the group of elements of $\text{GU}(4,\mathbb F_{3^2})$ with determinant 1 quotiented by multiples of the identity) acts symplectically on $X$. This action can be, again, extended by considering $G=\text{PGU}(4,\mathbb F_{3^2})$, which is $\text{GU}(4,\mathbb F_{3^2})$ quotiented by multiples of the identity. We can generate $G$ from $G_s$ by adding the multiplication of one variable by a primitive fourth root of unity. \\
The index of the extension is 4, which is maximal by Theorem \ref{NonSympIndex}. This non-symplectic group is also described in \parencite{Katsura_Kondo_Shimada_2014}.
\end{Possibility} 

\begin{Possibility}\label{S44index4}
    The group $G_s = 2^4:\mathfrak S_{3,3}$ acts symplectically on the superspecial $K3$ surface in characteristic 3 \parencite[Prop. 10.2]{ohashi2024finite}. \\
    The realization of Ohashi and Schütt is given by an action on the Kummer surface $X = \textrm{Km}(E_0\times E_0)$, where $E_0$ is the (unique) supersingular elliptic curve in characteristic 3 with $j$-invariant 0 \parencite[section 11.3]{ohashi2024finite}. Note that such a Kummer surface has Artin invariant 1. We obtain two different fibrations that can be seen on the affine model given by the equation
    \[
    (t^3-t) y^2=(x^3-x).
    \]
    The fibrations map to $\mathbb P^1_t$ and $\mathbb P^1_x$. Both fibrations have 4 singular fibres of type $\text{I}_0^*$ at 0,1,2 and the point at infinity. The symmetric group $\mathfrak S_4$ acts on $\mathbb P^1_t$ and $\mathbb P^1_x$ by permuting the singular fibres. Then a symplectic action on $X$ is given by the group $G_s$ generated by $\mathfrak A_{4,4}$ together with the switch of fibrations
    \[
    (t,x,y)\mapsto(x,t,-1/y).
    \]
    We would like to argue that there exists an index four extension of $G_s$, which is maximal in characteristic 3 due to Theorem \ref{NonSympIndex}.\\
    Some care is needed to realize the action of $\mathfrak S_4 \times \mathfrak S_4$ on $X$. For instance, when taking the permutation $\tau$ leaving $\mathbb P^1_t$ invariant and switching $1$ and $-1$ in $\mathbb P^1_x$, we change the value of $(x^3+x)$ by $-1$. This implies that such a action can only be realized by multiplying $y$ with a primitive fourth root of unity, which induces an action of the fourth root of unity on $H^{2,0}(X)$. This is also noted by Ohashi and Schütt in their realization. But although $\tau^2 = \id\in \mathfrak A_{4,4}$, we map $y$ to $-y$ so the action cannot be symplectic. This is similar to the situation in Realization \ref{AusführlichesBeispiel}.\\
    Explicitly we may take the group generated by the elements of $G_s$ and 
    \[
    \tau\colon(t,x,y)\mapsto(\tau(x),\tau(x),\zeta_4y),
    \]
    where $\zeta_4$ is a primitive fourth root of unity, $\tau\in\mathfrak S_{4,4}\setminus\mathfrak A_{4,4}$ and $\tau(x)$ and $\tau(t)$ are the respective induced actions on $\mathbb P_x$ and $\mathbb P_t$.   
    One can check that this generates a finite group $G$. The symplectic subgroup of $G$ is necessarily $G_s$ since $G_s$ is maximal in characteristic 3. Thus, we get the extension $(2^4:\mathfrak S_{3,3}).4$. 
  
\end{Possibility}

\begin{Possibility}\label{UnusalBehaviour}
    The group $G_s = 2^2.\mathfrak A_{4,4}$ acts symplectically on the supersingular $K3$ surface of Artin invariant 1 in characteristic 3. 
    The action of $G_s$ can be realized as the combination of an action of $T_{48}$ and $H_{192}$. We will abbreviate the discussion here and point to \parencite[section 11.3.2]{ohashi2024finite} for more details. The group $T_{48}$ acts on the resolution $\Tilde{Y}=X$ of 
    \[
    Y = V(x_3^2 = x_0x_1(x_0^4+x_1^4) +x_2^6)\subset \mathbb P(1,1,1,3).
    \]
    The variety $Y$ has 8 singularities of type $A_2$. After resolving them, we obtain 16 exceptional lines $E_i {^\pm}$, where each $E_i^{+}$ meets exactly one other exceptional line, $E_i^{-}$. Furthermore, there are 8 pairs of lines $l_j^{\pm}$. Each of these lines meets 3 exceptional divisors each (all with the same sign) and one other line $l_k^{\mp}$ with opposite sign.\\
    The combined action of $T_{48}$ and $H_{192}$ preserves this arrangement of lines $E_i^+$ and $l_j^-$ as well as the set $E_i^-$ and $l_j^+$ and we can arrange them into the singular fibres of the fibration 
    \[
    (t^3-t) y^2=(x^3-x).
    \]
    Note that this is the same fibration that has been discussed in Realization \ref{S44index4}.
    The automorphism obtained by multiplying $y$ with $-1$ also preserves both arrangements as it leaves the lines contained in the fibres invariant and consequently all other lines. As a consequence, it fixes the generators of the invariant lattice of $G_s$ in \parencite[Lemma 11.10]{ohashi2024finite}. Hence, we may extend $G_s$ with this non-symplectic automorphism and obtain an extension of index 2. 
\end{Possibility}
\begin{Lemma}
    Any extension of $2^2.\mathfrak A_{4,4}$ has non-symplectic index 1 or 2. This means that the extension in Realization \ref{UnusalBehaviour} has the highest possible non-symplectic index.
\end{Lemma}
\begin{proof}
    The co-invariant lattice $\NS_{G_s}$ is a negative definite rank 20 lattice  $8^{-2}_6 3^{-1}$. This implies that the invariant lattice $\NS^{G_s}$ is a rank 2 indefinite lattice.
    Since $\NS$ is a primitive extension of $\NS^{G_s}$ and $\NS_{G_s}$, we may conclude that we glue along a subgroup whose order is divisible by $8^2$. But apriori, it is not clear wether or not gluing along 3 occurs.
    \begin{Claim}
        $\det(\NS^{G_s})=3\cdot a$ for a suitable integer $a$ coprime to $3$.
    \end{Claim}
    \begin{proof}[Proof of Claim]
        Assume that the claim is false.
    The other possibility would be that $\det(\NS^{G_s})=3^3\cdot a$ for a suitable integer $a$ coprime to 3 and gluing along 3 occurs. 
    The subgroups $H_1\subset A_{\NS^{G_s}}$ and $H_2\subset A_{\NS_{G_s}}$ which we glue along have an order divisible by 3. 
    The 3-length of $\discgr{\NS^{G_s}}$ is bounded by the rank of $\NS^{G_s}$ and can be equal to 1 or 2. We may discard the possibility that the 3-length is 1 since one can calculate that the 3-Sylow group of the discriminant group would be isomorphic to $\mathbb Z /27 \mathbb Z$ and thus the primitive extension by this glue map would have a 3-Sylow group isomorphic to $\mathbb Z/9\mathbb Z$. \\
    If the 3-length of $\discgr{\NS^{G_s}}$ is 2, the 3-Sylow group $(\discgr{\NS^{G_s}})_3$ is isomorphic to $ \mathbb Z / 3\mathbb Z \times \mathbb Z /9\mathbb Z$ as a group. Furthermore, in order to glue into $\NS$, the subgroup we need to glue along is of the form $\langle 3a \rangle$, where $a \in \discgr{\NS^{G_s}}$ is an element of order 9. Otherwise, one can again calculate that the discriminant group of the primitive extension obtained by this gluing would have a discriminant group of the form $\mathbb Z/9\mathbb Z$.\\
    But such an element $3a$ always satisfies $q_{\discgr{\NS^{G_s}}}(a)\in \{0+2\mathbb Z\ 1+2\mathbb Z \}\subset \mathbb Q/2\mathbb Z$. Due to this, such a subgroup cannot be glued with the 3-part of $\discgr{\NS_{G_s}}$ due to the respective quadratic forms. Thus, the claim holds. 
    \end{proof}
   \noindent Thus, we may conclude that $\NS^{G_s}$ has the genus symbol $8^{-2}_23^{-1}$ and $\NS$ is a primitive extension of $\NS^{G_s}$ and $\NS_{G_s}$   . This implies that there are elements $x\in \NS^{G_s}$ and $y\in \NS_{G_s}$ which are not dually primitive in $\NS$. The discriminant group of $\NS$ can be generated by rational multiples of $x$ and $y$. We denote two such rational multiples of $x$ and $y$ as $a$ and $b$. We recall that any non-symplectic automorphism in an extension of $G_s$ still preserves the decomposition $\NS^{G_s}$ and $\NS_{G_s}$. Consequently, $a$ must be mapped to a multiple of $a$ and the same must hold for $b$.\\ 
   Assume that $g$ is a non-symplectic automorphism such that the group generated by $G_s$ and $g$ is a finite group. Denote the induced action of $g$ on $\discg$ as $f$. Then
    \[
    f\colon\discg \rightarrow \discg, \ f([a])\mapsto [n\cdot a], \ f([b])\mapsto  [m \cdot b] 
    \]
    for $n,m \in \{1,2 \}$. In other words, the class of $a$ (resp. $b$) in $\discg$ must be mapped to multiple of itself since $G/G_s$ is cyclic. Consequently, the non-symplectic index is at most 2.
\end{proof}

\begin{Possibility}
    The group $G_s=M_{21}.2_1\cong \text{PSL}(3,\mathbb F_{2^2}).2$ acts on the superspecial $K3$ surface in characteristic 2. A realization of the index 3 extension $\text{PGL}(3,\mathbb F_{2^2}).2$ is given in \parencite{Char2Dolgachev}. This non-symplectic index is optimal by Theorem 
\ref{NonSympIndex}.
\end{Possibility}

\subsection{Proof of Theorem \ref{wildcase}}\label{proofwild}
We may now collect the proof of Theorem \ref{wildcase} and briefly summarize the arguments.
\begin{proof}[Proof of Theorem \ref{wildcase}] 
    The maximal groups in characteristic $p\leq 11$ were classified by Ohashi and Schütt \parencite[Theorem 1.2]{ohashi2024finite}.
    The groups with trivial extensions are discussed in Section \ref{rk1noext}, whereas models with extensions of the highest possible non-symplectic index can be found in \ref{Char23}. There is a unique tame group in characteristic 11 that has no non-trivial extensions due to an analogous argument to the one in the proof of Proposition \ref{rk2222}. \\
    We note that the non-symplectic index is equal to 1 or $p+1$ for every group except for $2^2.\mathfrak A_{4,4}$ in characteristic 3, which is discussed in Realization \ref{UnusalBehaviour}. Here, the highest possible non-symplectic index is 2.
\end{proof}

\printbibliography

\end{document}